\documentclass[11pt]{amsart}

\usepackage[margin=1.15in]{geometry}
\usepackage{amsmath,amssymb,mathtools}
\usepackage{enumitem}
\usepackage{graphicx}
\usepackage{placeins}
\usepackage[colorlinks=true,citecolor=blue,linkcolor=blue,urlcolor=blue]{hyperref}
\usepackage[nameinlink,capitalise]{cleveref}

\graphicspath{{figures/}}

\newtheorem{theorem}{Theorem}[section]
\newtheorem{proposition}[theorem]{Proposition}
\newtheorem{lemma}[theorem]{Lemma}
\newtheorem{corollary}[theorem]{Corollary}

\theoremstyle{definition}
\newtheorem{definition}[theorem]{Definition}
\theoremstyle{remark}
\newtheorem{remark}[theorem]{Remark}

\newcommand{\R}{\mathbb R}
\newcommand{\HH}{\mathcal H}
\newcommand{\supp}{\operatorname{spt}}
\newcommand{\Tr}{\operatorname{Tr}}
\newcommand{\Phiop}{\mathcal Q_{\Phi}}
\newcommand{\Mphi}{\mathbf M_{\Phi}}
\newcommand{\llbracketset}[1]{\mathopen{[\![}#1\mathclose{]\!]}}

\title{A Halfspace Theorem for Global Anisotropic Perimeter Minimizers}
\author{Yang Yang}
\address{School of Mathematics, Hunan University}
\email{\tt yyang1@hnu.edu.cn}
\date{}

\begin{document}

\begin{abstract}
We prove an arbitrary-dimensional halfspace theorem for global minimizers of
a smooth uniformly elliptic anisotropic perimeter: if the nonempty boundary
of a minimizing set is contained in a halfspace, then the set itself is a
halfspace.  No evenness of the integrand or regularity of the minimizing
boundary is assumed.  The proof combines wall contact with a plane-peeling
argument for the complementary phase defect and a simultaneous second
blow-down.  We also give a direct, self-contained exterior-barrier proof of
the two-dimensional stationary anisotropic statement.  Its rigidity
conclusion is already covered by Bergner's earlier halfspace theorem; the
proof is retained because it works directly with the anisotropic
Euler--Lagrange operator.
\end{abstract}

\maketitle

\section{Introduction}

Halfspace theorems express a global rigidity principle: a minimal object that
remains on one side of a hyperplane should itself be flat.  For the Euclidean
area functional, this principle has two classical and complementary forms.
In hypersurface dimension two, the Hoffman--Meeks theorem says that a
connected proper minimal surface without boundary in $\R^3$ which is contained
in a halfspace must be a plane \cite{HoffmanMeeks1990}.  No stability or
minimizing assumption is required, and the proof is a maximum principle at
infinity based on catenoidal barriers.  Bergner subsequently proved a more
general halfspace theorem for proper negatively curved immersions satisfying
a uniform relation between their principal curvatures
\cite{Bergner2010}.\footnote{This reference was inadvertently omitted from
the first arXiv version.}  His framework includes critical points of weighted
area functionals and, through the curvature relation verified below, contains
the two-dimensional stationary anisotropic conclusion considered here.

At the minimizing level, Bombieri--Giusti established the corresponding
one-sided rigidity for oriented boundaries of least area in arbitrary
dimension \cite{BombieriGiusti1972}.  Their argument instead uses Harnack
theory on area-minimizing hypersurfaces and the harmonicity of Euclidean
coordinate functions.  Thus the classical theory naturally separates into a
stationary theory for two-dimensional surfaces and a theorem in every
dimension under the stronger assumption of global minimization.

An anisotropic surface energy assigns a direction-dependent cost to an
interface.  For a smooth oriented hypersurface $\Sigma^n\subset\R^{n+1}$ it
has the form
\[
 \mathcal A_\Phi(\Sigma)=\int_\Sigma\Phi(\nu_\Sigma)\,d\mathcal H^n.
\]
Here positive one-homogeneity of $\Phi$ makes the energy parametric, while
convexity and tangential uniform ellipticity provide the variational
ellipticity.  The choice $\Phi(z)=|z|$ recovers Euclidean area.  Geometrically,
$\Phi$ is the support function of its Wulff shape, the equilibrium shape in
the associated anisotropic isoperimetric problem.  Such energies arise as
continuum models for crystals, liquid drops, and interfaces with
direction-dependent surface tension; they also occur naturally in anisotropic
capillarity problems
\cite{FigalliMaggiPratelli2010,FigalliMaggi2011,DePhilippisMaggi2015}.
If $\Phi$ is even, the energy is independent of the choice of orientation.  A
non-even integrand instead distinguishes the two phases separated by the
interface, which is why oriented boundaries of sets are the natural objects
in the second part of this paper.

The first variation of $\mathcal A_\Phi$ gives the anisotropic mean-curvature
equation, a uniformly elliptic equation along the tangent space.  For graphs
it becomes a quasilinear divergence-form equation, while for general
interfaces its weak formulations belong to the theory of parametric elliptic
integrands, anisotropic varifolds, and finite-perimeter minimizers
\cite{SchoenSimonAlmgren1977,DePhilippisDeRosaGhiraldin2018,DeRosaTione2022}.
The theory is substantially less rigid than its Euclidean counterpart.  In
particular, the standard Euclidean monotonicity formula is not available for
a general anisotropy \cite{Allard1974}.  Consequently, arguments based on
blow-ups or blow-downs cannot simply import the usual area-ratio monotonicity;
one instead needs density estimates, compactness, barriers, and quantitative
flatness principles adapted to the integrand.

The Bernstein problem gives another indication of this difference.  For the
Euclidean area functional, entire minimal graphs are affine through
hypersurface dimension seven, with counterexamples from dimension eight onward
\cite{Simons1968,BombieriDeGiorgiGiusti1969}.  For general smooth uniformly
elliptic anisotropies, unconditional rigidity holds only through hypersurface
dimension three \cite{Jenkins1961,Simon1977} and fails in general in higher
dimensions \cite{Mooney2022,MooneyYang2024}, while the related theory of
minimizing cones exhibits further flexibility \cite{MooneyYang2021}.
Under additional quantitative assumptions on the integrand and the growth of
the graph, all-dimensional rigidity may nevertheless be recovered
\cite{DuYang2024}.

Boundary geometry supplies a different source of rigidity.  Edelen--Wang
proved that a Euclidean minimal graph over a proper convex domain is affine
when its boundary data are affine \cite{EdelenWang2022}.  In the anisotropic
setting, Du--Mooney--Yang--Zhu obtained the corresponding conclusion in every
dimension, in particular for graphs over a halfspace
\cite{DuMooneyYangZhu2025}.  More recently, Wang--Wei--Xia--Zhang proved that
an anisotropic minimal graph over a halfspace with anisotropic free boundary
is affine when its negative part has at most linear growth
\cite{WangWeiXiaZhang2026}.  These results show that one-sided or convex
geometry can restore rigidity beyond the dimensions in which the unrestricted
anisotropic Bernstein theorem is valid.

The purpose of the present paper is to study halfspace rigidity for genuinely
nongraphical anisotropic interfaces.  The main new result is the
arbitrary-dimensional theorem for global anisotropic perimeter minimizers.
The existing graph and free-boundary results do not address this setting, and
the Euclidean minimizing proof does not transfer formally: contact with a
supporting plane does not by itself exclude an additional phase defect in the
halfspace.  The treatment of this defect, together with the multiplicity-one
structure of a set boundary, is the central point of the argument.

For completeness, we also retain the two-dimensional stationary statement
and give a direct proof adapted to the anisotropic Euler--Lagrange operator.
Although its rigidity conclusion is already covered by Bergner's theorem, the
explicit global strict exterior barrier provides a useful PDE formulation of
the comparison argument and is logically independent of the minimizing proof.

We now introduce the common setting.  Let $n\ge1$ and let
$\Phi:\R^{n+1}\to[0,\infty)$ be a convex, positively
one-homogeneous integrand, with $\Phi(0)=0$; thus
$\Phi(tz)=t\Phi(z)$ for $t\ge0$ and $z\in\R^{n+1}$.  We assume throughout
that $\Phi\in C^3(\R^{n+1}\setminus\{0\})$ and that, for some
$\lambda\in(0,1]$,
\begin{equation}\label{old-eq:regularity-bounds}
 \lambda\le\Phi(\nu)\le\lambda^{-1},
 \qquad
 \|\Phi\|_{C^3(\mathbb S^n)}\le\lambda^{-1}.
\end{equation}
We also assume the tangential ellipticity bounds
\begin{equation}\label{old-eq:ellipticity}
  \lambda |\tau|^2
  \le D^2\Phi(\nu)[\tau,\tau]
  \le \lambda^{-1}|\tau|^2,
  \qquad \nu\in\mathbb S^n,\quad \tau\perp\nu.
\end{equation}
Let $\Sigma\subset\R^{n+1}$ be a smooth oriented $n$-dimensional hypersurface
with unit normal $\nu_\Sigma$.  For $\nu\in\mathbb S^n$, we denote by
\[
 D_T^2\Phi(\nu):=D^2\Phi(\nu)|_{\nu^\perp\times\nu^\perp}
\]
the restriction of the Hessian to $\nu^\perp$.  For $x\in\Sigma$ and
$X,Y\in T_x\Sigma$, our sign convention for the second fundamental form is
\[
 \mathrm{II}_{\Sigma}(X,Y)
 :=-\langle\nabla_X\nu_\Sigma,Y\rangle,
\]
where $\nabla$ is the Euclidean connection.  If
$\tau_1,\ldots,\tau_n$ is an orthonormal basis of $T_x\Sigma$, then
\begin{equation}\label{old-eq:tangential-contraction}
 \operatorname{tr}\bigl(D_T^2\Phi(\nu_\Sigma(x))\,
       \mathrm{II}_{\Sigma}\bigr)
 :=\sum_{i,j=1}^n D^2\Phi(\nu_\Sigma(x))[\tau_i,\tau_j]\,
       \mathrm{II}_{\Sigma}(\tau_i,\tau_j).
\end{equation}
The anisotropic area of $\Sigma$ in an open set $U$ is
\begin{equation}\label{old-eq:smooth-anisotropic-area}
 \mathcal A_\Phi(\Sigma;U)
 :=\int_{\Sigma\cap U}\Phi(\nu_\Sigma)\,d\HH^n.
\end{equation}
For a set $E$ of locally finite perimeter, we write $\partial^*E$ for its
reduced boundary and $\nu_E$ for its measure-theoretic outer unit normal,
with the convention
\begin{equation}\label{old-eq:BV-orientation-convention}
 -D\chi_E=\nu_E|D\chi_E|
 =\nu_E\HH^n\llcorner\partial^*E.
\end{equation}
Here $\mu\llcorner A$ is the restriction of a measure $\mu$ to a Borel set
$A$.  We define
\begin{equation}\label{old-eq:anisotropic-perimeter}
 P_\Phi(E;U)
 :=\int_{U\cap\partial^*E}\Phi(\nu_E)\,d\HH^n
 =\int_U\Phi(\nu_E)\,d|D\chi_E|.
\end{equation}

We use $B_r:=B_r(0)$.  For a possibly vector-valued Radon measure $\mu$, set
$\supp\mu:=\supp|\mu|$.  Equalities and inclusions between finite-perimeter
sets are understood up to $\mathcal L^{n+1}$-null sets unless stated
otherwise.

We first record the stationary statement in hypersurface dimension two.  Its
rigidity conclusion is not claimed as new.  To relate it explicitly to
Bergner's theorem, let $\tau_1,\tau_2$ be principal directions at a nonflat
point, let $\kappa_1,\kappa_2$ be the corresponding principal curvatures, and
set
\[
 a_i:=D^2\Phi(\nu_\Sigma)[\tau_i,\tau_i].
\]
Uniform ellipticity and stationarity give
\[
 \lambda\le a_i\le\lambda^{-1},
 \qquad
 a_1\kappa_1+a_2\kappa_2=0.
\]
Consequently, at a nonflat point the two curvatures have opposite signs;
after interchanging the indices, we may take $\kappa_1>0>\kappa_2$, and
\[
 \lambda^2
 \le -\frac{\kappa_1}{\kappa_2}
 =\frac{a_2}{a_1}
 \le\lambda^{-2}
\]
whenever the second fundamental form is nonzero.  Equivalently,
\[
 \kappa_1^2+\kappa_2^2
 \le (\lambda^2+\lambda^{-2})(-\kappa_1\kappa_2),
\]
and this inequality is trivial at a flat point.  This is the uniform
principal-curvature relation appearing in Bergner's framework.  His theorem
therefore yields the conclusion below, in the more general setting of proper
immersions.  We state the result in the present operator notation because
Sections~\ref{old-sec:barrier} and \ref{old-sec:proper-proof} give a
self-contained alternative proof.

When $\Phi$ is even, the functional is defined without choosing an orientation.
In the statement below, a
surface written as a subset of Euclidean space is understood to be embedded,
and ``proper'' means that $\Sigma\cap C$ is compact in $\Sigma$ for every
compact set $C\subset\R^3$.

\begin{theorem}[Anisotropic halfspace theorem in $\R^3$]
\label{old-thm:proper-halfspace}
Let $\Phi:\R^3\to[0,\infty)$ satisfy
\eqref{old-eq:regularity-bounds}--\eqref{old-eq:ellipticity}, with $n=2$, and
be even.  Let
$\Sigma\subset\R^3$ be a connected smooth oriented proper surface without
boundary satisfying
\begin{equation}\label{old-eq:theorem-stationarity}
 \operatorname{tr}\bigl(D_T^2\Phi(\nu_\Sigma)\,\mathrm{II}_\Sigma\bigr)=0
 \qquad\text{on }\Sigma.
\end{equation}
If $\Sigma$ is contained in a closed halfspace, then $\Sigma$ is an affine
plane.
\end{theorem}

\begin{remark}
The boundaryless hypothesis in \cref{old-thm:proper-halfspace} is essential.
Indeed, on the halfplane $\{x>0\}$ the function
\begin{equation}\label{old-eq:shifted-helicoid}
 u(x,y)=c\arctan\frac{y}{x+a},
 \qquad a>0,\quad c\ne0,
\end{equation}
defines a nonaffine minimal graph.  It is a shifted helicoid branch, is smooth
up to $\{x=0\}$, and satisfies $|u|<|c|\pi/2$; hence its graph is contained in
a slab, and therefore in a halfspace.  Its boundary trace
$u(0,y)=c\arctan(y/a)$ is nonlinear.  Including $\{x=0\}$ gives a surface
with noncompact boundary, whereas deleting that boundary destroys properness
in $\R^3$.  Moreover, for axially symmetric surface tensions the usual
helicoid can remain anisotropic minimal \cite{KuhnsPalmer2011}, so this
distinction is not peculiar to the isotropic integrand.
\end{remark}

We next pass from stationary surfaces to globally minimizing boundaries in
arbitrary dimension.  The natural objects here are multiplicity-one
boundaries of Caccioppoli sets.  This hypothesis is necessary for the theorem
in this form:
an area-minimizing integral current may be a sum of several calibrated
parallel hyperplanes, whereas the boundary of a set has unit multiplicity.

\begin{definition}
\label{old-def:perimeter-minimizer}
Let $U\subset\R^{n+1}$ be open.  A set $E\subset\R^{n+1}$ of locally finite
perimeter is a $\Phi$-perimeter minimizer in $U$ if
\begin{equation}\label{old-eq:global-minimizer}
  P_\Phi(E;W)\le P_\Phi(F;W)
\end{equation}
whenever $F$ is a set of locally finite perimeter, $W\subset U$ is open with
$\overline W$ compact in $U$, and $E\triangle F$ is compactly contained in
$W$.  It is a global
$\Phi$-perimeter minimizer when $U=\R^{n+1}$.
\end{definition}

\begin{theorem}[Halfspace theorem for globally minimizing boundaries]
\label{old-thm:minimizing-halfspace}
Let $\Phi$ satisfy
\eqref{old-eq:regularity-bounds}--\eqref{old-eq:ellipticity}, without any evenness
assumption, and let $E\subset\R^{n+1}$ be a global $\Phi$-perimeter minimizer.  If
\begin{equation}\label{old-eq:one-sided-support}
 \varnothing\ne\supp |D\chi_E|
 \subset\{x\cdot e\ge c\}
\end{equation}
for some $e\in\mathbb S^n$ and $c\in\R$, then $E$ agrees almost
everywhere with a halfspace.  In particular, $\supp |D\chi_E|$ is an affine
hyperplane.
\end{theorem}

\begin{remark}
The geometric content of \cref{old-thm:minimizing-halfspace} can be phrased as
follows.  The set $E$ represents one phase, and
$\Sigma_E:=\supp|D\chi_E|$ is the interface separating it from the other
phase.  The theorem asserts that if the entire interface is confined to one
side of a hyperplane, then it cannot bend,
oscillate, retain singularities, or contain additional components:
$\Sigma_E$ is an affine hyperplane and $E$ is one of the two halfspaces.  In
short, global anisotropic perimeter minimality together with one-sided
confinement forces complete flatness.  Notice that no graphical or regularity
hypothesis is imposed on $\Sigma_E$, and $\Phi$ need not be even.
\end{remark}

For completeness, we give a direct proof of
\cref{old-thm:proper-halfspace}.  Following the Hoffman--Meeks barrier
philosophy, we replace the exact catenoid by a strict comparison surface
adapted to the anisotropic graph operator.
Exact anisotropic catenoids and Delaunay-type surfaces are known under
additional symmetry assumptions on the surface tension
\cite{KoisoPalmer2005,KoisoPalmer2008}; these constructions do not supply the
required family for an arbitrary anisotropy and supporting direction.  Near a
supporting plane, the anisotropic graph operator linearizes to a
constant-coefficient elliptic equation in two variables, whose fundamental
solution is an elliptic logarithm.  Correcting this logarithm by a lower-order
inverse-radius term produces a strict subsolution on the complement of a
compact elliptic disk.  Its compact inner boundary and proper negative height,
together with the properness of $\Sigma$, make the first-contact argument
coercive and prevent contact from escaping to infinity.

For \cref{old-thm:minimizing-halfspace}, we use the De Philippis--Maggi density
estimates, compactness theorem, support convergence, and wall-contact maximum
principle \cite{DePhilippisMaggi2015}.  The central step is the following
rigidity statement for a minimizer touching its supporting wall.

\begin{proposition}[Wall rigidity for set minimizers]
\label{old-prop:wall-rigidity}
Let $H=\{x_{n+1}>0\}$ and $P=\partial H$.  Suppose that $F\subset H$ is a
global $\Phi$-perimeter minimizer and
$0\in\supp|D\chi_F|$.  Then $F=H$ almost everywhere.
\end{proposition}

The wall-contact principle first gives full measure-theoretic trace of $F$ on
$P$.  If the phase defect $K:=H\setminus F$ were nonempty, a localized
BV/current peeling argument would make $K$ a global minimizer for the reversed
integrand $\Phi^-(\nu):=\Phi(-\nu)$.  A simultaneous second blow-down of $F$
and $K$, together with strong convergence of their perimeter measures, would
produce limits $L$ and $G=H\setminus L$ satisfying the exact measure identity
\[
 |D\chi_G|=\HH^n\llcorner P+|D\chi_L|.
\]
Wall contact gives full trace to $L$ and hence zero trace to $G$; restricting
the identity to $P$ would therefore give $0=2\HH^n\llcorner P$.  This replaces
the global Harnack mechanism of
Bombieri--Giusti rather than attempting to reproduce it for anisotropic
coordinate functions.  Once the defect has been excluded, the flat blow-down
is converted into expanding one-sheet graphical regions by the interior
regularity theory for minimizing parametric elliptic integrands
\cite{SchoenSimonAlmgren1977}.

It would also be interesting to recover this last local graphical step by
combining Savin's improvement-of-flatness scheme \cite{Savin2010} with the
anisotropic one-sided measure estimate of Du--Mooney--Yang--Zhu
\cite{DuMooneyYangZhu2025}.  For possibly singular Caccioppoli minimizers, that
route would require an additional scale-sharp rough contact estimate.  It is
not used here.

We conclude the introduction by indicating the organization of the paper.
In \cref{old-sec:preliminaries} we fix orientation conventions and record the
compactness, wall-trace, and thin-slab regularity results used later.
Sections~\ref{old-sec:barrier} and \ref{old-sec:proper-proof} give the direct
exterior-barrier proof of the already known two-dimensional statement.
Sections~\ref{old-sec:current-minimality}--\ref{old-sec:global-proof} prove
the main theorem for global anisotropic perimeter minimizers.

\section{Preliminaries}\label{old-sec:preliminaries}

Three preliminary ingredients will be used in the two parts of the paper.
We first record the anisotropic Euler--Lagrange equation in parametric and
graphical form, including the effect of reversing orientation for a non-even
integrand.  We then collect the compactness, density, and wall-contact
consequences needed for perimeter minimizers.  The last ingredient is a
thin-slab theorem for absolutely minimizing oriented boundaries; before
stating it, we fix the corresponding gauge-mass convention.

For a smooth oriented hypersurface $\Sigma^n\subset\R^{n+1}$ with unit normal
$\nu$, the $\Phi$-minimal equation is
\begin{equation}\label{old-eq:parametric-EL}
  \operatorname{tr}\big(D_T^2\Phi(\nu)\,\mathrm{II}_\Sigma\big)=0.
\end{equation}
If $\Omega\subset\R^n$ is open and $\Sigma$ is the graph of
$u:\Omega\to\R$ with upward unit normal
$(-Du,1)/(1+|Du|^2)^{1/2}$, then
\eqref{old-eq:parametric-EL} becomes
\begin{equation}\label{old-eq:graph-EL}
  \Phiop[u]
  :=\operatorname{tr}\big(D^2\varphi(Du)D^2u\big)=0,
  \qquad \varphi(p)=\Phi(-p,1),\quad p\in\R^n.
\end{equation}

Set
\begin{equation}\label{old-eq:reversed-integrand-definition}
 \Phi^-(z):=\Phi(-z).
\end{equation}
Reversing the orientation replaces $\Phi$ by $\Phi^-$.  Notice that $\Phi^-$
satisfies the same regularity and ellipticity assumptions as $\Phi$.

We next record the compactness and wall-contact consequences of
De Philippis--Maggi that will be used in the blow-down arguments
\cite{DePhilippisMaggi2015}.

\begin{proposition}[Compactness and support convergence]
\label{old-prop:DPM-compactness}
Let $E_j$ be global $\Phi$-perimeter minimizers with locally uniformly
bounded perimeter, meaning that
$\sup_j|D\chi_{E_j}|(K)<\infty$ for every compact set $K$.  After passing to a
subsequence, there is a set $E$ of locally finite perimeter such that
\begin{equation}\label{old-eq:DPM-convergence}
  \chi_{E_j}\to\chi_E\quad\text{in }L^1_{\mathrm{loc}},
  \qquad
  D\chi_{E_j}\stackrel{*}{\rightharpoonup}D\chi_E,
  \qquad
  |D\chi_{E_j}|\stackrel{*}{\rightharpoonup}|D\chi_E|.
\end{equation}
The limit $E$ is a global minimizer, and
$\supp|D\chi_{E_j}|$ converges to $\supp|D\chi_E|$ locally in Hausdorff
distance.
\end{proposition}

Here $\stackrel{*}{\rightharpoonup}$ denotes local weak-star convergence of
Radon measures.

\begin{proof}
Apply \cite[Theorem~2.9, (2.57)--(2.58)]{DePhilippisMaggi2015} with the
degenerate container $H=\R^{n+1}$, the fixed integrand $\Phi$, and
$\Lambda=0$, first on a ball and then on an exhaustion of $\R^{n+1}$.
A diagonal subsequence gives the three convergences in
\eqref{old-eq:DPM-convergence} and preserves global minimality.  Notice that
the third convergence is convergence of the full Euclidean perimeter
measures; it is stronger than the lower semicontinuity supplied by ordinary
$BV$ compactness.

For completeness, the two directions of support convergence are precisely
the full-space versions of \cite[(2.60)--(2.61)]{DePhilippisMaggi2015}.
Indeed, points of $\supp|D\chi_E|$ are approximated by points of
$\supp|D\chi_{E_j}|$, while every convergent sequence
$x_j\in\supp|D\chi_{E_j}|$ has its limit in $\supp|D\chi_E|$.
The uniform lower density estimate
\cite[Lemma~2.8]{DePhilippisMaggi2015} makes these implications uniform on
compact subsets, and hence gives local Hausdorff convergence.
\end{proof}

If $F\subset H$ has locally finite perimeter and $P=\partial H$, then
$\Tr_PF$ denotes the one-sided $BV$ trace of $\chi_F$ on $P$, taken from
inside $H$.

\begin{proposition}[Wall-trace maximum principle]\label{old-prop:wall-trace}
Let $R>0$, let $F\subset H=\{x_{n+1}>0\}$, and set $P=\partial H$.
Suppose that $F$ is a $\Phi$-perimeter minimizer in $B_{2R}$ in the sense of
\cref{old-def:perimeter-minimizer}, with ambient competitors not constrained to
lie in $H$.  If
$0\in\supp|D\chi_F|$, then
\begin{equation}\label{old-eq:full-wall-trace}
  \Tr_PF=1\qquad\HH^n\text{-a.e. on }P\cap B_R.
\end{equation}
\end{proposition}

\begin{proof}
After translating and dilating by $R$, this is
\cite[Lemma~2.13]{DePhilippisMaggi2015}.  Its hypothesis (2.93) is exactly
the ambient minimality assumed here.  Since the alternative
$0\notin\supp|D\chi_F|$ is excluded, that lemma gives full trace on the
wall; restricting back to $P\cap B_R$ gives
\eqref{old-eq:full-wall-trace}.
\end{proof}

\begin{remark}
The hypothesis in \cref{old-prop:wall-trace} is ambient minimality: competitors
may cross $P$, and no free-boundary or obstacle constraint is imposed.  This
is precisely the form of condition (2.93) in Lemma~2.13 of
\cite{DePhilippisMaggi2015} needed below.
The corresponding strict maximum principle in the isotropic area-minimizing
setting goes back to \cite{Simon1987}.
\end{remark}

The remaining regularity input is formulated for absolutely minimizing
oriented boundaries.  Since reversing orientation changes the energy when
$\Phi$ is not even, we first make precise the mass used in that statement.

For $v\in BV_{\mathrm{loc}}(\R^{n+1};\mathbb Z)$, let
$\llbracketset{v}$ denote the top-dimensional current with integer
multiplicity $v$ and the standard orientation.
For a set $E$ of locally finite perimeter,
$\llbracketset{E}:=\llbracketset{\chi_E}$ is the
multiplicity-one current of integration over $E$, and we set
\begin{equation}\label{old-eq:boundary-current}
  T_E:=\partial\llbracketset{E}.
\end{equation}
Here $\partial$ denotes current boundary.  If
$T=\vec T\,\|T\|$ is the polar decomposition of a locally integral
codimension-one current, then $\|T\|$ is its mass measure, $\vec T$ is its
unit oriented $n$-vector, and $\supp T:=\supp\|T\|$.  We fix the
Hodge-star identification between oriented $n$-vectors and normal vectors and
define the normal vector measure $\mathbf n(T):=(*\vec T)\,\|T\|$, with the
choice of sign for which
\begin{equation}\label{old-eq:normal-vector-measure}
 \mathbf n(\partial\llbracketset{v})=-Dv,
 \qquad
 \mathbf n(T_E)=\nu_E|D\chi_E|=-D\chi_E.
\end{equation}
For locally integral codimension-one currents $T$ and $S$, we have
$\mathbf n(T+S)=\mathbf n(T)+\mathbf n(S)$,
$\mathbf n(-T)=-\mathbf n(T)$, and $|\mathbf n(T)|=\|T\|$.

If $\mu$ is an $\R^{n+1}$-valued Radon measure and $A$ is Borel, define
\begin{equation}\label{old-eq:vector-gauge}
 \mathcal F_\Phi(\mu;A)
 :=\int_A\Phi\left(\frac{d\mu}{d|\mu|}\right)\,d|\mu|,
\end{equation}
where $d\mu/d|\mu|$ is the polar density.  For a locally integral $n$-current
$T$, set
\begin{equation}\label{old-eq:oriented-gauge-mass}
 \mu_T^\Phi(A):=\mathcal F_\Phi(\mathbf n(T);A),
 \qquad
 \Mphi(T;A):=\mu_T^\Phi(A).
\end{equation}
In particular, our orientation convention gives
\begin{equation}\label{old-eq:set-current-mass}
  \Mphi(T_E;A)=P_\Phi(E;A).
\end{equation}
For a non-even integrand this is an oriented gauge mass, rather than an
unoriented elliptic mass.

An integral $n$-cycle is an integral $n$-current $S$ satisfying $\partial S=0$.
A locally integral $n$-current $T$ is said to be absolutely $\Phi$-mass minimizing
in an open set $U$ if
\begin{equation}\label{old-eq:local-absolute-mass-minimality}
 \Mphi(T;W)\le \Mphi(T+S;W)
\end{equation}
whenever $W\subset\subset U$ is open and $S$ is a compactly supported integral
$n$-cycle with $\supp S\subset W$.

With this convention in place, we turn to thin-slab regularity.  For $r>0$
and $\nu\in\mathbb S^n$, set
\begin{equation}\label{old-eq:cylinder}
 \mathcal C_r(\nu)
 :=\{x\in\R^{n+1}:
 |x-(x\cdot\nu)\nu|<r,\ |x\cdot\nu|<r\}.
\end{equation}
This is the open cylinder centered at the origin, with axis $\R\nu$,
tangential radius $r$, and height $2r$.

\begin{proposition}[Interior thin-slab regularity]
\label{old-prop:epsilon-regularity}
There exist constants
\[
 \varepsilon_0>0,
 \qquad \beta\in(0,1),
 \qquad C<\infty,
\]
depending only on $n$ and $\lambda$, with the following property.  Let $E$ be
a set of locally finite perimeter, let $r>0$, let
$x_0\in\supp|D\chi_E|$, and suppose that the oriented boundary current
$T_E$ is absolutely $\Phi$-mass minimizing in
$B_{2r}(x_0)$.  If $Q$ is an affine hyperplane through $x_0$ and
\begin{equation}\label{old-eq:thin-slab-hypothesis}
 \supp|D\chi_E|\cap B_r(x_0)
 \subset\{x:\operatorname{dist}(x,Q)<\varepsilon_0r\},
\end{equation}
then
\[
 M:=\supp|D\chi_E|\cap B_{\beta r}(x_0)
\]
is a connected embedded $C^2$ hypersurface, with
\[
 \overline M\setminus M
 \subset\partial B_{\beta r}(x_0).
\]
In particular, there are no singular points or additional components of the
support in $B_{\beta r}(x_0)$.  On $M$, the notation $\nu_E$ denotes the
continuous classical unit normal extending the measure-theoretic outer
normal.  It satisfies
\begin{equation}\label{old-eq:SSA-normal-estimate}
 |\nu_E(x)-\nu_E(y)|
 \le C\frac{|x-y|}{r},
 \qquad x,y\in M.
\end{equation}
\end{proposition}

\begin{proof}
Translate $x_0$ to the origin and dilate by $r^{-1}$.  Under the
Hodge-star identification fixed above, the corresponding parametric
integrand on oriented simple $n$-vectors is
\[
 \mathcal F(\xi)=\Phi(*\xi).
\]
Its ellipticity constants are controlled by $\lambda$, and no symmetry
under $\xi\mapsto-\xi$ is required.  The rescaled current belongs to the
class of absolutely minimizing oriented boundaries considered in
\cite[Part~I, Theorem~1.2]{SchoenSimonAlmgren1977}, and
\eqref{old-eq:thin-slab-hypothesis} is precisely condition~(37) there.
The conclusion of that theorem identifies the \emph{entire} support in
the smaller ball, not only the component containing the origin: it is a
connected $C^2$ hypersurface whose relative boundary lies on the sphere.
Estimate~(38) of the same theorem gives the asserted estimate for the
unit normal.  Scaling back proves the proposition.
\end{proof}

\begin{corollary}[One-sheet graphical convergence]
\label{old-prop:strict-flat-graphical}
Let $E_j$ be $\Phi$-perimeter minimizers in $B_2$, suppose that $T_{E_j}$ is
absolutely $\Phi$-mass minimizing in $B_2$, and assume
$0\in\supp|D\chi_{E_j}|$.  Let $H$ be a halfspace with measure-theoretic outer
unit normal $\nu_H$, let $0\in P:=\partial H$, and suppose that
\begin{equation}\label{old-eq:strict-BV-flatness}
 \chi_{E_j}\to\chi_H
 \quad\text{in }L^1(B_2),
 \qquad
 D\chi_{E_j}\stackrel{*}{\rightharpoonup}D\chi_H,
 \qquad
 |D\chi_{E_j}|\stackrel{*}{\rightharpoonup}|D\chi_H|.
\end{equation}
Then there exists $\theta_*=\theta_*(n,\lambda)>0$ such that, for all
sufficiently large $j$, there is a function
$u_j:P\cap B_{\theta_*}\to\R$ of class $C^2$ for which
\begin{equation}\label{old-eq:flat-graph-representation}
 \supp|D\chi_{E_j}|\cap\mathcal C_{\theta_*}(\nu_H)
 =\{y+u_j(y)\nu_H:y\in P\cap B_{\theta_*}\}.
\end{equation}
Moreover,
\begin{equation}\label{old-eq:C1-flat-convergence}
 \|u_j\|_{C^1(P\cap B_{\theta_*})}\longrightarrow0,
 \qquad
 \sup_{\supp|D\chi_{E_j}|\cap
       \mathcal C_{\theta_*}(\nu_H)}
 |\nu_{E_j}-\nu_H|\longrightarrow0.
\end{equation}
Thus the whole support in the smaller cylinder is one sheet, not merely one
regular component.
\end{corollary}

\begin{proof}
Write
\[
 S_j:=\supp|D\chi_{E_j}|,
 \qquad
 \mu_j:=|D\chi_{E_j}|,
 \qquad
 \mu:=|D\chi_H|=\HH^n\llcorner P,
\]
and let $\pi(x):=x-(x\cdot\nu_H)\nu_H$ be the orthogonal projection onto
$P$.  The lower density estimate for perimeter minimizers, together with
$\mu_j\stackrel{*}{\rightharpoonup}\mu$, implies
\begin{equation}\label{old-eq:flat-support-convergence}
 S_j\longrightarrow P
 \qquad\text{locally in Hausdorff distance in }B_2.
\end{equation}
Indeed, a point of $S_j\cap B_{3/2}$ that remains a fixed positive
distance from $P$ carries a fixed amount of $\mu_j$ in a ball disjoint
from $P$, contradicting the convergence of $\mu_j$.  Conversely, if a
ball centered at a point of $P\cap B_{3/2}$ were disjoint from $S_j$
along a subsequence, the lower semicontinuity inequality on that open ball
would contradict the positivity of $\mu$ there.  Consequently, for all
sufficiently large $j$,
\[
 S_j\cap B_1
 \subset\{x:\operatorname{dist}(x,P)<\varepsilon_0\}.
\]
Applying \cref{old-prop:epsilon-regularity} with $r=1$ shows that
\[
 M_j:=S_j\cap B_\beta
\]
is a connected embedded $C^2$ hypersurface and
\begin{equation}\label{old-eq:normal-Lipschitz-local}
 |\nu_{E_j}(x)-\nu_{E_j}(y)|\le C|x-y|,
 \qquad x,y\in M_j.
\end{equation}

We next upgrade convergence of the oriented perimeter measures to uniform
convergence of the normals on a smaller ball.  If
$\eta\in C_c(B_\beta)$ is nonnegative, then
$-D\chi_{E_j}=\nu_{E_j}\mu_j$ and the three convergences in
\eqref{old-eq:strict-BV-flatness} give
\begin{align*}
 \int \eta|\nu_{E_j}-\nu_H|^2\,d\mu_j
 &=2\int\eta\,d\mu_j
   +2\nu_H\cdot\int\eta\,dD\chi_{E_j}\\
 &\longrightarrow
   2\int\eta\,d\mu
   +2\nu_H\cdot\int\eta\,dD\chi_H=0,
\end{align*}
because $D\chi_H=-\nu_H\mu$.  Together with
\eqref{old-eq:normal-Lipschitz-local} and the lower density estimate, this
implies
\begin{equation}\label{old-eq:uniform-normal-convergence}
 \sup_{M_j\cap B_{\beta/2}}|\nu_{E_j}-\nu_H|\longrightarrow0.
\end{equation}
Indeed, otherwise there are $\delta>0$ and
$x_j\in M_j\cap B_{\beta/2}$ at which the displayed difference is at least
$\delta$.  For fixed
$s<\min\{\beta/4,\delta/(2C)\}$, it is then at least $\delta/2$ on
$M_j\cap B_s(x_j)$, whose $\mu_j$-measure is at least $cs^n$.  This
contradicts the preceding integral convergence with a cutoff equal to one
on $B_{3\beta/4}$.

Set $\theta_*:=\beta/16$.  By
\eqref{old-eq:flat-support-convergence}, after increasing $j$ if necessary,
\[
 d_j:=\sup\{\operatorname{dist}(x,P):x\in S_j\cap B_1\}<\theta_*,
 \qquad d_j\longrightarrow0.
\]
Every point of
\[
 N_j:=M_j\cap\{x:|\pi(x)|<2\theta_*\}
\]
then lies in $B_{\beta/2}$.  By
\eqref{old-eq:uniform-normal-convergence}, the restriction of $\pi$ to
$N_j$ is a local diffeomorphism.  On each connected component it is also
proper over the disk $P\cap B_{2\theta_*}$: if the projections of a
sequence remain in a compact subset of that disk, the height bound
$|x\cdot\nu_H|\le d_j$ gives a convergent subsequence whose limit can meet
neither $\partial B_\beta$ nor the lateral boundary
$|\pi(x)|=2\theta_*$.  Connected components are closed in $N_j$, so this
limit remains in the same component.  Thus the image of each component is
both open and closed in $P\cap B_{2\theta_*}$, and each restriction is a covering map.
The disk is simply connected, so every component is the graph of one $C^2$
function over the whole disk.

There can be only one component.  Let
$\omega_n:=\HH^n(B_1^n)$.  The constancy theorem and the multiplicity-one
structure of the set boundary give
\[
 \mu_j\llcorner M_j=\HH^n\llcorner M_j.
\]
Because $d_j<\theta_*$, each graphical component
contributes at least
\[
 \int_{P\cap B_{\theta_*}}\sqrt{1+|Du|^2}\,d\HH^n
 \ge \omega_n\theta_*^n
\]
to $\mu_j(\mathcal C_{\theta_*}(\nu_H))$.  The cylinder is a continuity set
for $\mu=\HH^n\llcorner P$, and hence
\[
 \mu_j(\mathcal C_{\theta_*}(\nu_H))
 \longrightarrow
 \mu(\mathcal C_{\theta_*}(\nu_H))
 =\omega_n\theta_*^n.
\]
Two components would contradict this convergence.  At least one exists
because $0\in S_j$, so there is exactly one.

If $x\in S_j\cap\mathcal C_{\theta_*}(\nu_H)$, then
$|x|<\sqrt2\,\theta_*<\beta$ and $|\pi(x)|<\theta_*$, so $x\in N_j$.
Conversely, the part over $P\cap B_{\theta_*}$ of the unique graphical
component lies in this cylinder because its height is bounded by
$d_j<\theta_*$.  Denote its graphing function there by $u_j$.  These two
inclusions prove \eqref{old-eq:flat-graph-representation} for the whole
support.  Moreover,
$\|u_j\|_{L^\infty}\le d_j\to0$, while
\eqref{old-eq:uniform-normal-convergence} and
\[
 |Du_j|=
 \frac{|\pi(\nu_{E_j})|}{|\nu_{E_j}\cdot\nu_H|}
\]
give $\|Du_j\|_{L^\infty}\to0$.  This proves
\eqref{old-eq:C1-flat-convergence}.
\end{proof}

\section{Strict logarithmic exterior barriers in
\texorpdfstring{$\R^3$}{R3}}
\label{old-sec:barrier}

In this section, we begin the promised direct proof of the two-dimensional
statement.  Although its rigidity conclusion follows from Bergner's theorem,
the construction below is self-contained and works directly with the
anisotropic graph operator.  We construct strict exterior barriers with
$n=2$ and supporting normal $e_3$.  Their leading term is the
logarithmic fundamental solution of the linearized anisotropic graph operator;
an inverse-radius correction makes the comparison inequality strict.  After
dilation, the resulting exterior graphs converge to the supporting plane away
from a shrinking inner boundary while their ends remain vertically proper.
These are precisely the properties used in the first-contact argument of the
next section.

Recall the graph integrand and set
\begin{equation}\label{old-eq:graph-integrand}
  \varphi(p):=\Phi(-p,1),\qquad A:=D^2\varphi(0),
  \qquad p\in\R^2.
\end{equation}
The ellipticity assumption makes $A$ symmetric positive definite.  We
introduce the elliptic radius
\begin{equation}\label{old-eq:elliptic-radius}
  \rho(x):=(x^TA^{-1}x)^{1/2},
  \qquad x\in\R^2.
\end{equation}

\begin{lemma}[Strict anisotropic logarithmic barrier]
\label{old-lem:strict-barrier}
There exist constants $C_*=C_*(\Phi)>0$, $a_0=a_0(\Phi)>0$, and
$c_*=c_*(\Phi)>0$ such that, for every $a\in(0,a_0)$, the function
\begin{equation}\label{old-eq:barrier}
  b(x)
  :=-a\log\rho(x)+C_*a^2\big(\rho(x)^{-1}-1\big),
  \qquad \rho(x)\ge1,
\end{equation}
satisfies
\begin{align}
 b&=0 &&\text{on }\{\rho=1\},\label{old-eq:barrier-boundary}\\
 b&<0,\qquad \lim_{\rho\to\infty}b=-\infty,
   &&\text{on }\{\rho>1\},\label{old-eq:barrier-proper}\\
 \Phiop[b]&\ge c_*a^2\rho^{-3}>0
   &&\text{on }\{\rho>1\}.\label{old-eq:barrier-strict}
\end{align}
Moreover, $|Db|$ is uniformly small when $a_0$ is chosen sufficiently small.
\end{lemma}

\begin{proof}
For matrices $M$ and $N$, write
\[
 M:N:=\operatorname{tr}(M^TN).
\]
Set $B:=A^{-1}$, $r:=\rho(x)$, and $q:=Bx$.  Since $r^2=x^TBx$, direct
differentiation gives
\begin{align}
 D\log r&=\frac{q}{r^2},&
 D^2\log r&=\frac{B}{r^2}-2\frac{q\otimes q}{r^4},
 \label{old-eq:log-derivatives}\\
 D(r^{-1})&=-\frac{q}{r^3},&
 D^2(r^{-1})&=-\frac{B}{r^3}+3\frac{q\otimes q}{r^5}.
 \label{old-eq:inverse-radius-derivatives}
\end{align}
Moreover, $A:B=2$ and $q^TAq=r^2$.  Contracting
\eqref{old-eq:log-derivatives} and
\eqref{old-eq:inverse-radius-derivatives} with $A$ therefore yields
\begin{equation}\label{old-eq:linear-barrier-identities}
 A:D^2\log\rho=0,
 \qquad
 A:D^2(\rho^{-1})=\rho^{-3}.
\end{equation}
Equivalently, after the linear change of variables $y=A^{-1/2}x$, these
identities become
\[
 \Delta_y\log|y|=0,
 \qquad
 \Delta_y|y|^{-1}=|y|^{-3}
\]
in two dimensions.

The eigenvalues of $A$ are bounded above and below in terms of the
ellipticity data.  Hence $\rho$ is comparable to the Euclidean radius, and
differentiating the preceding formulas gives, for $k=1,2$,
\begin{equation}\label{old-eq:radial-derivative-decay}
 |D^k\log\rho|\le C_k\rho^{-k},
 \qquad
 |D^k(\rho^{-1})|\le C_k\rho^{-k-1}.
\end{equation}
Consequently,
\begin{align}
 |Db|&\le C\bigl(a\rho^{-1}+C_*a^2\rho^{-2}\bigr),
 \label{old-eq:Db-bound}\\
 |D^2b|&\le C\bigl(a\rho^{-2}+C_*a^2\rho^{-3}\bigr).
 \label{old-eq:D2b-bound}
\end{align}

Choose $r_0>0$ so that $D^2\varphi$ is Lipschitz on
$\overline{B_{r_0}(0)}$.  Once $a_0$ is sufficiently small,
\eqref{old-eq:Db-bound} ensures that $|Db|\le r_0$ on $\{\rho\ge1\}$.
Thus, for some $L=L(\Phi)$,
\begin{equation}\label{old-eq:Hessian-Taylor-bound}
 \|D^2\varphi(Db)-A\|\le L|Db|.
\end{equation}
Here and below $\|\cdot\|$ is the operator norm.  Since
\eqref{old-eq:linear-barrier-identities} gives
$A:D^2b=C_*a^2\rho^{-3}$, we obtain
\begin{align}
 \Phiop[b]
 &=C_*a^2\rho^{-3}
   +\bigl(D^2\varphi(Db)-A\bigr):D^2b \notag\\
 &\ge
 \bigl(C_*-K_0-K_1C_*a-K_2C_*^2a^2\bigr)a^2\rho^{-3},
 \label{old-eq:quantitative-barrier-error}
\end{align}
where $K_0,K_1,K_2$ depend only on $\Phi$.  Indeed,
\eqref{old-eq:Db-bound}--\eqref{old-eq:D2b-bound} imply
\[
 |Db|\,|D^2b|
 \le
 \bigl(K_0a^2+K_1C_*a^3+K_2C_*^2a^4\bigr)\rho^{-3},
\]
where we used $\rho\ge1$.

Now fix $C_*:=2(K_0+1)$ and then decrease $a_0$ so that
\[
 K_1C_*a_0+K_2C_*^2a_0^2\le1
\]
and the small-gradient condition above also holds.  It follows from
\eqref{old-eq:quantitative-barrier-error} that
\[
 \Phiop[b]\ge \frac{C_*}{2}a^2\rho^{-3},
\]
so we may take $c_*=C_*/2$.

Finally, if $\rho>1$, then both $-a\log\rho$ and
$C_*a^2(\rho^{-1}-1)$ are negative.  This proves
\eqref{old-eq:barrier-boundary}--\eqref{old-eq:barrier-proper}, while the
uniform smallness of $Db$ follows from \eqref{old-eq:Db-bound}.
\end{proof}

For $t>0$, define
\begin{equation}\label{old-eq:scaled-barrier}
  b_t(x):=t\,b(x/t),
  \qquad \rho(x)\ge t.
\end{equation}

\begin{lemma}[Scaled barrier convergence]
\label{old-lem:barrier-scaling}
The scaled barriers satisfy
\begin{equation}\label{old-eq:scaled-strictness}
  \Phiop[b_t](x)=t^{-1}\Phiop[b](x/t)>0,
\end{equation}
and
\begin{equation}\label{old-eq:punctured-plane-convergence}
  b_t\to0
  \quad\text{in }C^2_{\mathrm{loc}}(\R^2\setminus\{0\})
  \quad\text{as }t\downarrow0.
\end{equation}
For each fixed $t$, the height function $b_t$ is proper and tends to
$-\infty$ along its end.
\end{lemma}

\begin{proof}
The elliptic radius is positively one-homogeneous, so
\[
 \rho(x/t)=t^{-1}\rho(x).
\]
Thus $b(x/t)$ is defined precisely when $\rho(x)\ge t$; this is the reason
for the domain in \eqref{old-eq:scaled-barrier}.  Differentiating gives
\[
 Db_t(x)=Db(x/t),
 \qquad
 D^2b_t(x)=t^{-1}D^2b(x/t).
\]
Therefore
\[
 \Phiop[b_t](x)
 =t^{-1}\Phiop[b](x/t)
 \ge c_*a^2t^2\rho(x)^{-3}>0.
\]

For the convergence assertion, write explicitly
\begin{equation}\label{old-eq:scaled-barrier-expansion}
 b_t(x)
 =-at\log\frac{\rho(x)}{t}
   +C_*a^2\left(\frac{t^2}{\rho(x)}-t\right).
\end{equation}
On a compact set $K\subset\subset\R^2\setminus\{0\}$, the functions $\rho$,
$\rho^{-1}$, and their derivatives are uniformly bounded.  Since
$t|\log t|\to0$, \eqref{old-eq:scaled-barrier-expansion} and its first two
derivatives converge uniformly to zero on $K$.  This proves
\eqref{old-eq:punctured-plane-convergence}.

For fixed $t$, the right-hand side of
\eqref{old-eq:scaled-barrier-expansion} is strictly decreasing as a function
of $\rho$, since
\[
 \frac{d}{d\rho}b_t
 =-\frac{at}{\rho}-\frac{C_*a^2t^2}{\rho^2}<0,
\]
and it tends to $-\infty$ as $\rho\to\infty$.  This proves the last
assertion.
\end{proof}

The geometry summarized in \cref{old-fig:exterior-barrier} is exactly what is
used below.  In the left panel, the artificial inner boundary
$\{\rho=t\}$ shrinks to the origin, the graph becomes flat on compact subsets
away from it, and the exterior end remains vertically proper.  The right
panel anticipates the fixed disk $\mathcal D=\{\rho\le R_0\}$ and the vertical
translate $w=b_t+\delta_t$ introduced in the proof below.  The level sets
$\{\rho=t\}$ and $\{\rho=R_0\}$ are concentric.  Once $t$ is fixed, the
passage from $b_t$ to $w$ changes only the height of the barrier.

\begin{figure}[htbp]
  \centering
  \includegraphics[width=.98\textwidth]{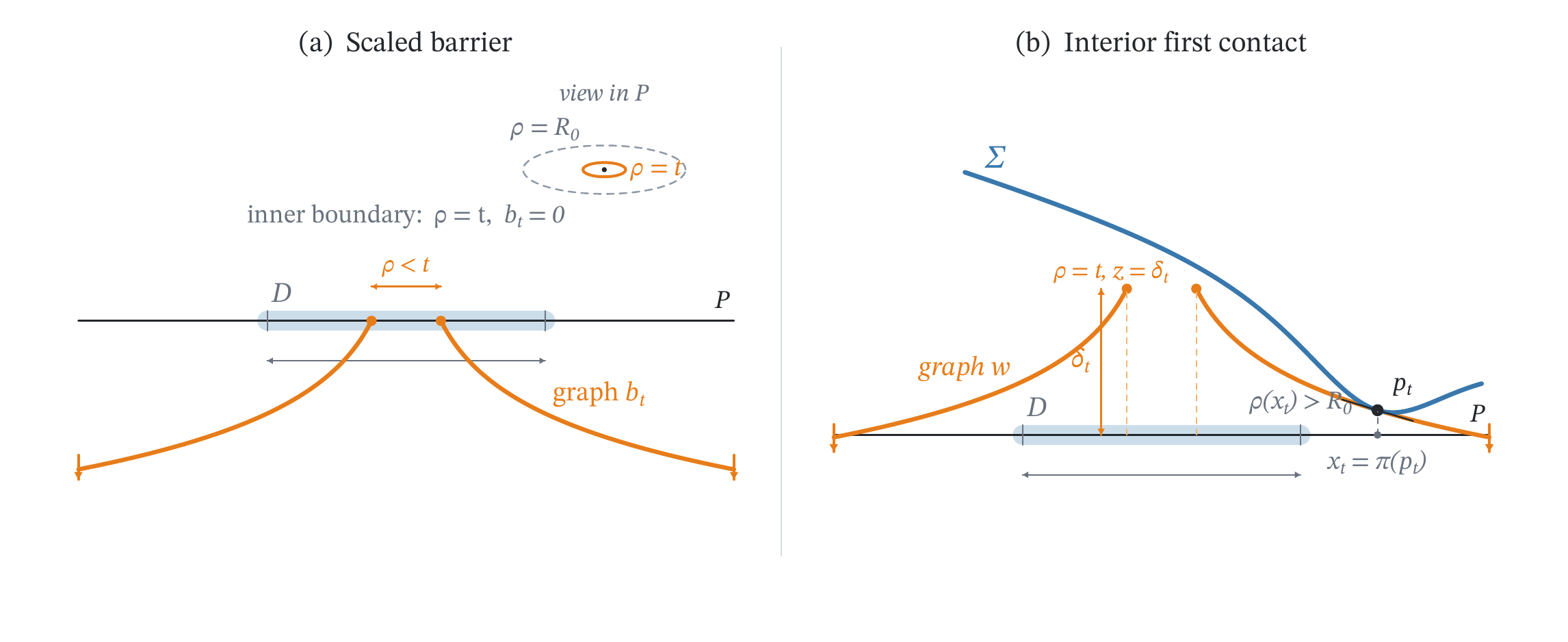}
  \caption{Scaled logarithmic barrier and interior first contact.}
  \label{old-fig:exterior-barrier}
\end{figure}
\FloatBarrier

\section{A direct exterior-barrier proof of the two-dimensional statement}
\label{old-sec:proper-proof}

In this section, we complete the direct proof by using the strict exterior
graphs constructed above to turn a hypothetical nonplanar surface into an
interior first-contact contradiction.
Properness first gives
a positive separation from the supporting plane over a fixed disk and then
makes the difference between the surface height and a translated barrier
coercive.  Its negative minimum therefore occurs at an interior point, where
stationarity is incompatible with the strict barrier inequality.

\begin{proof}[Proof of \cref{old-thm:proper-halfspace}]
Apply an orthogonal map $Q$ to the coordinates and replace the integrand by
$\Phi_Q(\xi):=\Phi(Q^{-1}\xi)$.  Relabel the transformed surface and
integrand as $\Sigma$ and $\Phi$, and then translate.  Write points of
$\R^3$ as $(x,z)\in\R^2\times\R$, set $P:=\{z=0\}$, and identify $P$ with
$\R^2$.  Suppose that
\begin{equation}\label{old-eq:surface-normalization}
 \Sigma\subset\{z\ge0\},\qquad \inf_\Sigma z=0.
\end{equation}
If $\Sigma$ meets $P$, the strong maximum principle makes the contact
set open in $\Sigma$ \cite{SolomonWhite1989}; it is closed as well, so
connectedness gives $\Sigma\subset P$.  A surface without boundary is locally
open in $P$, while properness makes it closed in $P$; hence $\Sigma=P$.
Thus a nonplanar counterexample would satisfy $\Sigma\subset\{z>0\}$.

Let $\pi(x,z):=x$ be the horizontal projection onto the identified plane
$P\simeq\R^2$.  Choose $R_0>0$ and let
$\mathcal D:=\{x\in P:\rho(x)\le R_0\}$.  The set
\[
 \Sigma\cap\pi^{-1}(\mathcal D)\cap\{0\le z\le1\}
\]
is compact.  Since it is disjoint from $P$, the height has a positive minimum
there; after taking the minimum with $1$, we obtain $m_0>0$ such that
\begin{equation}\label{old-eq:positive-height-over-disk}
 z(p)\ge m_0\qquad\text{whenever }p\in\Sigma\text{ and }\pi(p)\in\mathcal D.
\end{equation}
If the displayed compact set is empty, take $m_0=1$.

Choose $\varepsilon\in(0,m_0/4)$ and a point $q\in\Sigma$ with
$z(q)<\varepsilon$.  Then $\pi(q)\notin\mathcal D$.  By
\eqref{old-eq:punctured-plane-convergence}, for all sufficiently small $t$ the
artificial inner boundary is contained in $\mathcal D$ and
\begin{equation}\label{old-eq:negative-test-point}
  b_t(\pi(q))>z(q)-\varepsilon.
\end{equation}
On the closed exterior portion
\begin{equation}\label{old-eq:exterior-preimage}
  \Sigma_t:=\{p\in\Sigma:\rho(\pi(p))\ge t\},
\end{equation}
define
\begin{equation}\label{old-eq:coercive-contact-function}
  f_t(p):=z(p)-\varepsilon-b_t(\pi(p)).
\end{equation}
On the artificial boundary $\{\rho=t\}$, one has $b_t=0$ and
$f_t\ge m_0-\varepsilon>0$, whereas \eqref{old-eq:negative-test-point} gives
$f_t(q)<0$.

Every sublevel set $\{p\in\Sigma_t:f_t(p)\le c\}$, $c\in\R$, is compact.
It is empty if $c+\varepsilon<0$.  Otherwise, if $f_t(p)\le c$, then
\[
 0\le z(p)\le c+\varepsilon,
 \qquad
 0\le-b_t(\pi(p))\le c+\varepsilon.
\]
Since $-b_t(x)\to+\infty$ strictly as $\rho(x)\to\infty$, the second
inequality bounds $\rho(\pi(p))$.  Thus the sublevel is a closed subset of
the intersection of $\Sigma$ with a compact ambient cylinder, and is compact
by properness.  Hence $f_t$ attains a negative minimum at an interior point
$p_t\in\Sigma_t$.
Equivalently, the minimum chooses the vertical placement shown in the right
panel of \cref{old-fig:exterior-barrier}: neither the artificial inner
boundary nor spatial infinity can be a contact location.

Set $x_t:=\pi(p_t)$ and
\begin{equation}\label{old-eq:touching-vertical-shift}
 \delta_t:=\varepsilon+f_t(p_t)
 =z(p_t)-b_t(x_t),
 \qquad
 w(x):=b_t(x)+\delta_t.
\end{equation}
Since $f_t(p_t)<0$, one has $\delta_t<\varepsilon$.  On the artificial
boundary $\{\rho=t\}$ we have $f_t>0$, whereas $f_t(p_t)<0$; hence
$\rho(x_t)>t$ and $b_t(x_t)<0$.  Since $z(p_t)>0$, it follows that
$\delta_t>0$.  Thus the vertical translation lifts the artificial inner
boundary to $\{\rho=t,z=\delta_t\}$ without changing its horizontal
footprint.  Moreover, every $p\in\Sigma$ with
$t\le\rho(\pi(p))\le R_0$ satisfies
$w(\pi(p))\le\delta_t<\varepsilon<m_0\le z(p)$.  Consequently
$x_t\notin\mathcal D$, so $\rho(x_t)>R_0$ and the contact occurs on the smooth
descending exterior part of the barrier, as shown in the right panel of
\cref{old-fig:exterior-barrier}.
Since $\rho(x_t)>R_0>t$, the barrier is smooth in a neighborhood of $x_t$.
Consider the ambient function
\[
 F_t(x,z):=z-\varepsilon-b_t(x).
\]
The restriction of $F_t$ to $\Sigma$ has a local minimum at $p_t$.
Consequently,
\[
 \nabla F_t(p_t)=(-Db_t(x_t),1)
\]
is perpendicular to $T_{p_t}\Sigma$.  Its vertical component is nonzero,
so the horizontal projection restricts to a local diffeomorphism on
$\Sigma$ near $p_t$.  We may therefore write this branch as $z=v(x)$ near
$x_t$.

Because $\Phi$ is even, reversing the chosen orientation of this local
branch does not change its stationarity equation.  We may hence use the
upward normal and conclude from \eqref{old-eq:graph-EL} that
\[
 \Phiop[v](x_t)=0.
\]
The definition of $\delta_t$ and the minimality of $f_t(p_t)$ give, for $x$
near $x_t$,
\[
 v(x)-w(x)=f_t((x,v(x)))-f_t(p_t)\ge0.
\]
Thus
\[
 v(x_t)=w(x_t),\qquad
 Dv(x_t)=Dw(x_t),\qquad
 D^2v(x_t)-D^2w(x_t)\ge0.
\]
Set
\[
 B_t:=D^2\varphi(Dw(x_t))=D^2\varphi(Dv(x_t)).
\]
The small-gradient choice in \cref{old-lem:strict-barrier}, together with
the ellipticity of $\Phi$, makes $B_t$ positive definite.  Therefore
\[
 \Phiop[w](x_t)-\Phiop[v](x_t)
 =B_t:\bigl(D^2w(x_t)-D^2v(x_t)\bigr)\le0.
\]
A vertical translation does not change the derivatives of the barrier, so
\[
 \Phiop[w](x_t)=\Phiop[b_t](x_t)>0.
\]
We have obtained
\begin{equation}\label{old-eq:first-contact-contradiction}
 0<\Phiop[w](x_t)\le\Phiop[v](x_t)=0,
\end{equation}
which is the desired contradiction.
\end{proof}

\begin{remark}\label{old-rem:non-even-2d}
For a non-even integrand, the operator at a contact point depends on the actual
orientation of the touching branch.  An extension to a properly embedded
separating boundary would require a proof that the phase fixes this
orientation throughout the contact construction, using $\Phi^-$ when the
orientation is reversed.  We do not claim that extension without such an
additional lemma.
\end{remark}

\begin{remark}
Theorem \ref{old-thm:proper-halfspace} is a one-surface halfspace theorem.  A
two-surface strong halfspace theorem additionally requires a separation
argument and is not claimed here.
\end{remark}

\section{Set minimizers and absolutely minimizing boundaries}
\label{old-sec:current-minimality}

In this section, we prove that a global set minimizer has an absolutely
minimizing oriented boundary.  The mass and orientation conventions were fixed
in \cref{old-sec:preliminaries}.  The point is that an arbitrary compactly
supported cycle perturbation need not itself be the boundary of a
Caccioppoli-set competitor.  Subadditivity, monotone integer-$BV$ truncation,
and codimension-one filling remove this obstruction.

\begin{lemma}[Subadditivity and orientation reversal]
\label{old-lem:asymmetric-gauge}
Let $\mu$ and $\eta$ be vector-valued Radon measures.  Then, as Radon
measures,
\begin{equation}\label{old-eq:gauge-subadditivity}
 \mathcal F_\Phi(\mu+\eta;\cdot)
 \le\mathcal F_\Phi(\mu;\cdot)+\mathcal F_\Phi(\eta;\cdot).
\end{equation}
Consequently, for any locally integral $n$-currents $T$ and $S$,
\begin{equation}\label{old-eq:anisotropic-current-triangle}
 \mu_{T+S}^\Phi\le\mu_T^\Phi+\mu_S^\Phi.
\end{equation}
Moreover, for $\Phi^-$ defined in
\eqref{old-eq:reversed-integrand-definition},
\begin{equation}\label{old-eq:reversed-current-mass}
 \mathcal F_\Phi(-\mu;\cdot)=\mathcal F_{\Phi^-}(\mu;\cdot),
 \qquad
 \mu_{-T}^\Phi=\mu_T^{\Phi^-},
 \qquad
 \Mphi(-T_E;A)=P_{\Phi^-}(E;A)
\end{equation}
for every Borel set $A$.
No evenness of $\Phi$ is required.
\end{lemma}

\begin{proof}
Take $\sigma=|\mu|+|\eta|$ and write
$a=d\mu/d\sigma$ and $b=d\eta/d\sigma$.  Convexity and positive
one-homogeneity imply
$\Phi(a+b)\le\Phi(a)+\Phi(b)$.  Integrating on an arbitrary Borel set proves
\eqref{old-eq:gauge-subadditivity}.  The reversal identities follow directly
from $\Phi(-z)=\Phi^-(z)$.
\end{proof}

\begin{lemma}[Monotone truncation]
\label{old-lem:monotone-truncation}
Let $U\subset\R^{n+1}$ be open, let
$v\in BV_{\mathrm{loc}}(U;\mathbb Z)$, and define
$\tau:\R\to[0,1]$ and $G\subset U$ by
\begin{equation}\label{old-eq:truncation}
 \tau(s):=\min\{1,\max\{0,s\}\},
 \qquad G:=\{v\ge1\}.
\end{equation}
Then $\tau\circ v=\chi_G$ almost everywhere.  There is a Borel function
$\vartheta:U\to[0,1]$ such that
\begin{equation}\label{old-eq:monotone-chain-rule}
 D(\tau\circ v)=\vartheta\,Dv.
\end{equation}
Consequently, as Radon measures,
\begin{equation}\label{old-eq:truncation-energy}
 \mu_{T_G}^\Phi
 \le\mu_{\partial\llbracketset{v}}^\Phi.
\end{equation}
\end{lemma}

\begin{proof}
Since $v$ is integer valued, $\tau\circ v=\chi_{\{v\ge1\}}$.  The scalar
$BV$ chain rule \cite[Theorem~3.99]{AmbrosioFuscoPallara2000}, applied to the
nondecreasing $1$-Lipschitz map $\tau$, gives
\eqref{old-eq:monotone-chain-rule}.  The coefficient $\vartheta$ is
nonnegative: on the jump part it is the difference quotient
\[
 \vartheta
 =\frac{\tau(v^+)-\tau(v^-)}{v^+-v^-}
\]
when $v^+\ne v^-$, and on the diffuse part it is the corresponding
nonnegative approximate derivative.  Thus the polar direction is never
reversed.  More explicitly,
\[
 \mathbf n(T_G)=-D(\tau\circ v)
 =\vartheta(-Dv)
 =\vartheta\,\mathbf n(\partial\llbracketset{v}).
\]
Positive one-homogeneity now gives, as measures,
\[
 \mu_{T_G}^{\Phi}
 =\vartheta\,\mu_{\partial\llbracketset{v}}^{\Phi}
 \le\mu_{\partial\llbracketset{v}}^{\Phi},
\]
which is \eqref{old-eq:truncation-energy}.  This computation is valid without
assuming that $\Phi$ is even.
\end{proof}

\begin{lemma}[Codimension-one filling]
\label{old-lem:integer-filling}
Let $S$ be a compactly supported integral $n$-cycle in $\R^{n+1}$.  Then
there exists $g\in BV(\R^{n+1};\mathbb Z)$ with compact essential support
such that
\begin{equation}\label{old-eq:integer-BV-filling}
 S=\partial\llbracketset{g},
 \qquad
 \supp g\subset\operatorname{conv}(\supp S).
\end{equation}
Here $\supp g$ denotes the essential support of $g$.
The normalization that $g=0$ on the unbounded component of
$\R^{n+1}\setminus\supp S$ determines $g$ uniquely.
\end{lemma}

\begin{proof}
The assertion is immediate if $S=0$.  Otherwise choose
$p\in\operatorname{conv}(\supp S)$ and let $R:=p*S$ denote the cone current
with vertex $p$ over $S$.  Since $\partial S=0$, the cone formula gives
$\partial R=S$, while
$\supp R\subset\operatorname{conv}(\supp S)$.  Since $R$ is a compactly
supported integral current of top dimension in $\R^{n+1}$, the
representation theorem for top-dimensional integral currents gives an
integer-valued $g\in L^1(\R^{n+1})$ such that
$R=\llbracketset{g}$; see \cite[4.1.28]{Federer1969}.  Moreover,
$\partial R=S$ has finite mass and, under this representation,
$\|\partial R\|=|Dg|$.  Hence $g\in BV(\R^{n+1};\mathbb Z)$.  Since $R$
vanishes outside $\operatorname{conv}(\supp S)$, $g$ may be chosen to vanish
there, and thus has compact essential support.

If $g_1$ and $g_2$ are two such normalized potentials, then
$\partial\llbracketset{g_1-g_2}=0$.  The constancy theorem
\cite[4.1.31]{Federer1969} makes $g_1-g_2$ constant almost everywhere, and
the normalization forces that constant to be zero.
\end{proof}

\begin{proposition}[Absolute minimality of the boundary current]
\label{old-prop:set-to-current}
If $E$ is a global $\Phi$-perimeter minimizer, then $T_E$ is absolutely
$\Phi$-mass minimizing.  More precisely, if $W\subset\subset\R^{n+1}$ is open and
$S$ is an integral $n$-cycle with $\supp S\subset W$, then
\begin{equation}\label{old-eq:absolute-current-minimality}
  \Mphi(T_E;W)\le\Mphi(T_E+S;W).
\end{equation}
\end{proposition}

\begin{proof}
By \cref{old-lem:integer-filling}, write
$S=\partial\llbracketset{g}$ with
$g\in BV(\R^{n+1};\mathbb Z)$ of compact essential support.  Choose a ball
$B$ such that
\[
 \overline W\cup\operatorname{conv}(\supp S)\subset\subset B.
\]
Then $\supp g\subset\subset B$.  Set
\begin{equation}\label{old-eq:integer-potential}
 v:=\chi_E+g,
 \qquad G:=\{v\ge1\}.
\end{equation}
Then $\partial\llbracketset{v}=T_E+S$, while $G=E$ outside $\supp g$.
Thus $G$ is an admissible set competitor.  By
\cref{old-lem:monotone-truncation},
\begin{equation}\label{old-eq:current-truncation}
 \Mphi(T_G;B)
 \le\Mphi(\partial\llbracketset{v};B)
 =\Mphi(T_E+S;B).
\end{equation}
Set minimality and \eqref{old-eq:set-current-mass} give
\[
 \Mphi(T_E;B)=P_\Phi(E;B)
 \le P_\Phi(G;B)=\Mphi(T_G;B).
\]
Combining the last two inequalities proves
\[
 \Mphi(T_E;B)\le\Mphi(T_E+S;B).
\]
Since $S=0$ on $B\setminus W$, the two currents, and hence their
$\Phi$-mass measures, agree there.  Subtracting this common finite
contribution yields \eqref{old-eq:absolute-current-minimality}, which is
precisely absolute minimality in the sense of
\eqref{old-eq:local-absolute-mass-minimality}.
\end{proof}

\begin{remark}
All masses above are localized to the bounded ball $B$.  No subtraction of
infinite global masses is involved.  The monotonicity in
\eqref{old-eq:monotone-chain-rule} is the point that preserves orientation for
a non-even integrand.
\end{remark}

\section{Plane peeling and the second blow-down}
\label{old-sec:peeling}

In this section, we prove the wall-rigidity proposition by isolating the phase
defect left behind by a minimizer with full trace on its supporting plane.  The trace
formula splits off the plane exactly, while the oriented minimality established
above shows that the defect minimizes the reversed integrand.  A simultaneous
blow-down preserves this measure splitting for both phases.  If the defect
survived, wall contact would give it full trace as well, producing an impossible
double contribution on the supporting plane.

\begin{lemma}[Wall-defect splitting]
\label{old-lem:wall-defect-splitting}
Let $H:=\{x_{n+1}>0\}$, let $P:=\partial H$, and let
$\nu_H=-e_{n+1}$ be the outer normal of $H$.  Suppose that $F\subset H$ has
locally finite perimeter and
\begin{equation}\label{old-eq:F-full-trace}
 \Tr_PF=1\qquad\HH^n\text{-a.e. on }P.
\end{equation}
Set $K:=H\setminus F$.  Then $\Tr_PK=0$,
\begin{equation}\label{old-eq:defect-no-wall-mass}
 |D\chi_K|(P)=0,
 \qquad T_F=T_H-T_K,
\end{equation}
and, as Radon measures on $\R^{n+1}$,
\begin{align}
 |D\chi_F|
 &=\HH^n\llcorner P+|D\chi_K|,
 \label{old-eq:perimeter-splitting}\\
 \mu_{T_F}^\Phi
 &=\Phi(\nu_H)\HH^n\llcorner P+\mu_{-T_K}^\Phi
 =\Phi(\nu_H)\HH^n\llcorner P+\mu_{T_K}^{\Phi^-}.
 \label{old-eq:mass-splitting}
\end{align}
\end{lemma}

\begin{proof}
For a finite-perimeter set $E_0\subset H$, extended by zero to $\R^{n+1}$,
the zero-extension trace formula is
\begin{equation}\label{old-eq:zero-extension-trace}
 D\chi_{E_0}\llcorner P
 =-(\Tr_PE_0)\nu_H\HH^n\llcorner P;
\end{equation}
this is the specialization to a halfspace of the trace identities
\cite[(2.10)--(2.14)]{DePhilippisMaggi2015}.  Since
$\chi_F+\chi_K=\chi_H$, the trace of $K$ is zero and
\eqref{old-eq:zero-extension-trace} gives $|D\chi_K|(P)=0$.
Differentiating the identity $\chi_F=\chi_H-\chi_K$ gives
$D\chi_F=D\chi_H-D\chi_K$, equivalently $T_F=T_H-T_K$.  The measure
$D\chi_H$ is concentrated on $P$, whereas $|D\chi_K|(P)=0$; hence the two
terms are mutually singular.  This proves the Euclidean splitting
\eqref{old-eq:perimeter-splitting}.  Applying the oriented gauge to the two
mutually singular normal measures proves \eqref{old-eq:mass-splitting}.
\end{proof}

The geometric content of \cref{old-lem:wall-defect-splitting} is shown in
\cref{old-fig:defect-splitting}.  The blue and orange arrows along
$\partial K$ indicate $\nu_K$ and $-\nu_K=\nu_F$, respectively, while the
black arrows indicate $\nu_H$ along $P$.  Thus the two boundary pieces in the
figure represent $T_H$ and $-T_K$ in the identity $T_F=T_H-T_K$.  The sign of
$-T_K$ is essential when $\Phi$ is not even.  No smoothness, connectedness,
or graphical structure is assumed for $K$.

\begin{figure}[htbp]
  \centering
  \includegraphics[width=.48\textwidth]{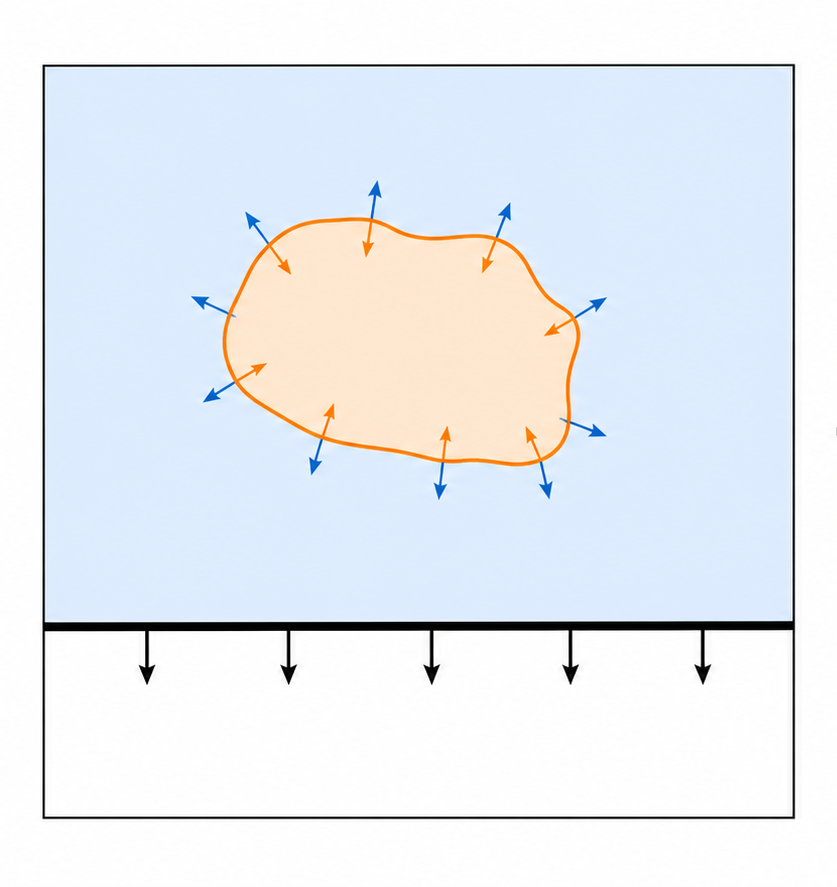}
  \caption{Plane peeling and the oriented defect.}
  \label{old-fig:defect-splitting}
\end{figure}
\FloatBarrier

\begin{proposition}[Minimality of the defect]
\label{old-prop:defect-minimality}
In the setting of \cref{old-lem:wall-defect-splitting}, suppose in addition
that $F$ is a global $\Phi$-perimeter minimizer.  Then $K$ is a global
$\Phi^-$-perimeter minimizer, and $-T_K$ is absolutely minimizing for
$\Mphi$.
\end{proposition}

\begin{proof}
We first prove ambient set minimality, allowing the competitor to cross the
wall.  Let $W\subset\subset\R^{n+1}$ be open and let $L$ be a set of locally finite
perimeter such that $K\triangle L\subset\subset W$; no inclusion $L\subset H$ is
assumed.  Define the integer-valued $BV$ function and its monotone truncation
by
\[
 v:=\chi_H-\chi_L,
 \qquad
 G:=\{v\ge1\}=H\setminus L.
\]
Then $G=F$ outside $W$, so $G$ is an admissible competitor for $F$.  Moreover,
\[
 -Dv=-D\chi_H+D\chi_L
     =\mathbf n(T_H)+\mathbf n(-T_L).
\]
The monotone truncation lemma and gauge subadditivity therefore give
\begin{align}
 P_\Phi(G;W)
 &=\Mphi(T_G;W)\notag\\
 &\le \mathcal F_\Phi(-Dv;W)\notag\\
 &\le P_\Phi(H;W)+P_{\Phi^-}(L;W).
 \label{old-eq:defect-competitor-upper}
\end{align}
On the other hand, the exact splitting
\eqref{old-eq:mass-splitting} and the minimality of $F$ yield
\begin{equation}\label{old-eq:defect-competitor-lower}
 P_\Phi(H;W)+P_{\Phi^-}(K;W)
 =P_\Phi(F;W)
 \le P_\Phi(G;W).
\end{equation}
Combining \eqref{old-eq:defect-competitor-upper} and
\eqref{old-eq:defect-competitor-lower} and cancelling the finite plane term
proves
\[
 P_{\Phi^-}(K;W)\le P_{\Phi^-}(L;W).
\]
Thus $K$ is a global $\Phi^-$-perimeter minimizer with respect to arbitrary
ambient competitors.  Applying \cref{old-prop:set-to-current} to the pair
$(K,\Phi^-)$ shows that $T_K$ is absolutely $\Phi^-$-mass minimizing.
Finally, the reversal identities in
\eqref{old-eq:reversed-current-mass}, together with the bijection
$S\mapsto-S$ on compactly supported integral cycles, show equivalently that
$-T_K$ is absolutely $\Phi$-mass minimizing.
\end{proof}

\begin{lemma}[Simultaneous blow-down]
\label{old-lem:common-splitting-limit}
Let $F\subset H$ and $K=H\setminus F$ satisfy
\cref{old-lem:wall-defect-splitting}, and suppose that $F$ and $K$ are global
minimizers for $\Phi$ and $\Phi^-$, respectively.  For $s>0$,
write $s^{-1}A:=\{x/s:x\in A\}$.  Given $s_j>0$ with $s_j\to\infty$, set
\begin{equation}\label{old-eq:simultaneous-blowdown}
 F_j:=s_j^{-1}F,
 \qquad K_j:=s_j^{-1}K.
\end{equation}
After passing to one common subsequence, there exist global minimizers
$G,L\subset H$ for $\Phi$ and $\Phi^-$, respectively, such that
\begin{align}
 \chi_{K_j}&\to\chi_L\quad\text{in }L^1_{\mathrm{loc}},&
 D\chi_{K_j}&\stackrel{*}{\rightharpoonup}D\chi_L,&
 |D\chi_{K_j}|&\stackrel{*}{\rightharpoonup}|D\chi_L|,
 \label{old-eq:K-strict-limit}\\
 \chi_{F_j}&\to\chi_G\quad\text{in }L^1_{\mathrm{loc}},&
 D\chi_{F_j}&\stackrel{*}{\rightharpoonup}D\chi_G,&
 |D\chi_{F_j}|&\stackrel{*}{\rightharpoonup}|D\chi_G|
 \label{old-eq:F-strict-limit}
\end{align}
and
\begin{equation}\label{old-eq:joint-limits}
 G=H\setminus L,
 \qquad
 |D\chi_G|=\HH^n\llcorner P+|D\chi_L|.
\end{equation}
\end{lemma}

\begin{proof}
Dilations preserve global minimality.  They also preserve the full-trace
identity and its coefficient exactly, because $H$ and $P$ are cones through
the origin.  The density estimates give locally uniform perimeter bounds for
both sequences.  Apply \cref{old-prop:DPM-compactness} first to $K_j$ and
then, along the resulting subsequence, to $F_j$.  Equivalently, these are the
strict measure convergences in
\cite[Theorem~2.9, (2.58)]{DePhilippisMaggi2015}; support persistence is
contained in (2.60)--(2.61) there.  A single diagonal subsequence therefore
gives all three convergences in both
\eqref{old-eq:K-strict-limit} and \eqref{old-eq:F-strict-limit}.  Since
$\chi_{F_j}=\chi_H-\chi_{K_j}$ at every scale, their $L^1_{\rm loc}$ limits
satisfy $G=H\setminus L$.

At every scale, \cref{old-lem:wall-defect-splitting} gives
$|D\chi_{F_j}|=\HH^n\llcorner P+|D\chi_{K_j}|$.  Therefore, for every
$\zeta\in C_c(\R^{n+1})$,
\begin{align*}
 \int\zeta\,d|D\chi_G|
 &=\lim_j\int\zeta\,d|D\chi_{F_j}|\\
 &=\int_P\zeta\,d\HH^n
   +\lim_j\int\zeta\,d|D\chi_{K_j}|\\
 &=\int\zeta\,d\bigl(\HH^n\llcorner P+|D\chi_L|\bigr).
\end{align*}
This proves the Radon-measure identity in \eqref{old-eq:joint-limits} without
subtracting infinite masses.
\end{proof}

\begin{lemma}[Trace incompatibility at the wall]
\label{old-lem:double-wall}
Let $L\subset H$, let $G:=H\setminus L$, and suppose that on a relatively
open set $\Gamma\subset P$,
\begin{equation}\label{old-eq:L-full-trace-local}
 \Tr_PL=1\qquad\HH^n\text{-a.e. on }\Gamma.
\end{equation}
Then
\begin{equation}\label{old-eq:wall-measure-restrictions}
 |D\chi_L|\llcorner\Gamma=\HH^n\llcorner\Gamma,
 \qquad
 |D\chi_G|\llcorner\Gamma=0.
\end{equation}
Consequently, the identity
\begin{equation}\label{old-eq:double-wall-identity}
 |D\chi_G|=\HH^n\llcorner P+|D\chi_L|
\end{equation}
cannot hold when $\HH^n(\Gamma)>0$.
\end{lemma}

\begin{proof}
For any finite-perimeter set $E_0\subset H$, the zero-extension formula
\eqref{old-eq:zero-extension-trace} implies the scalar measure identity
\[
 |D\chi_{E_0}|\llcorner\Gamma
 =(\Tr_PE_0)\HH^n\llcorner\Gamma.
\]
Taking $E_0=L$ and using \eqref{old-eq:L-full-trace-local} gives the first
identity in \eqref{old-eq:wall-measure-restrictions}.  Since
$\Tr_PG=1-\Tr_PL=0$, taking $E_0=G$ gives the second.  Restricting
\eqref{old-eq:double-wall-identity} to $\Gamma$ would yield
$0=2\HH^n\llcorner\Gamma$, which is impossible.
\end{proof}

\begin{proof}[Proof of \cref{old-prop:wall-rigidity}]
By \cref{old-prop:wall-trace}, applied in every ball centered at the origin,
$\Tr_PF=1$ almost everywhere on $P$.  Set $K:=H\setminus F$.  By
\cref{old-lem:wall-defect-splitting,old-prop:defect-minimality}, $K$ has zero
trace on $P$ and is a global $\Phi^-$-perimeter minimizer.

Suppose that $\mathcal L^{n+1}(K)>0$.  Then $\supp|D\chi_K|$ is nonempty:
otherwise $\chi_K$ would be constant almost everywhere, while $K\subset H$
rules out the constant one and positive measure rules out the constant zero.
Moreover,
\begin{equation}\label{old-eq:defect-away-from-wall}
 \supp|D\chi_K|\cap P=\varnothing.
\end{equation}
Indeed, a wall contact, after tangential translation, would allow
\cref{old-prop:wall-trace} to be applied in a small ball and would give full
trace of $K$ on an open subset of $P$, contradicting $\Tr_PK=0$.

Choose $q\in\supp|D\chi_K|\cap H$ and let $s_j\to\infty$.  Apply
\cref{old-lem:common-splitting-limit} to $F_j=s_j^{-1}F$ and
$K_j=s_j^{-1}K$.  Since $q/s_j\in\supp|D\chi_{K_j}|$ and $q/s_j\to0$,
support persistence gives
\begin{equation}\label{old-eq:L-touches-wall}
 0\in\supp|D\chi_L|.
\end{equation}
The wall-trace maximum principle for the global $\Phi^-$-minimizer $L$ gives
\begin{equation}\label{old-eq:L-full-trace}
 \Tr_PL=1\qquad\HH^n\text{-a.e. on }P\cap B_1.
\end{equation}

The simultaneous dilation and the resulting multiplicity obstruction are
depicted in \cref{old-fig:double-wall}.  In the last panel the solid and
dashed lines represent two measures supported on the same wall, not two
separated parallel walls.

\begin{figure}[htbp]
  \centering
  \includegraphics[width=.98\textwidth]{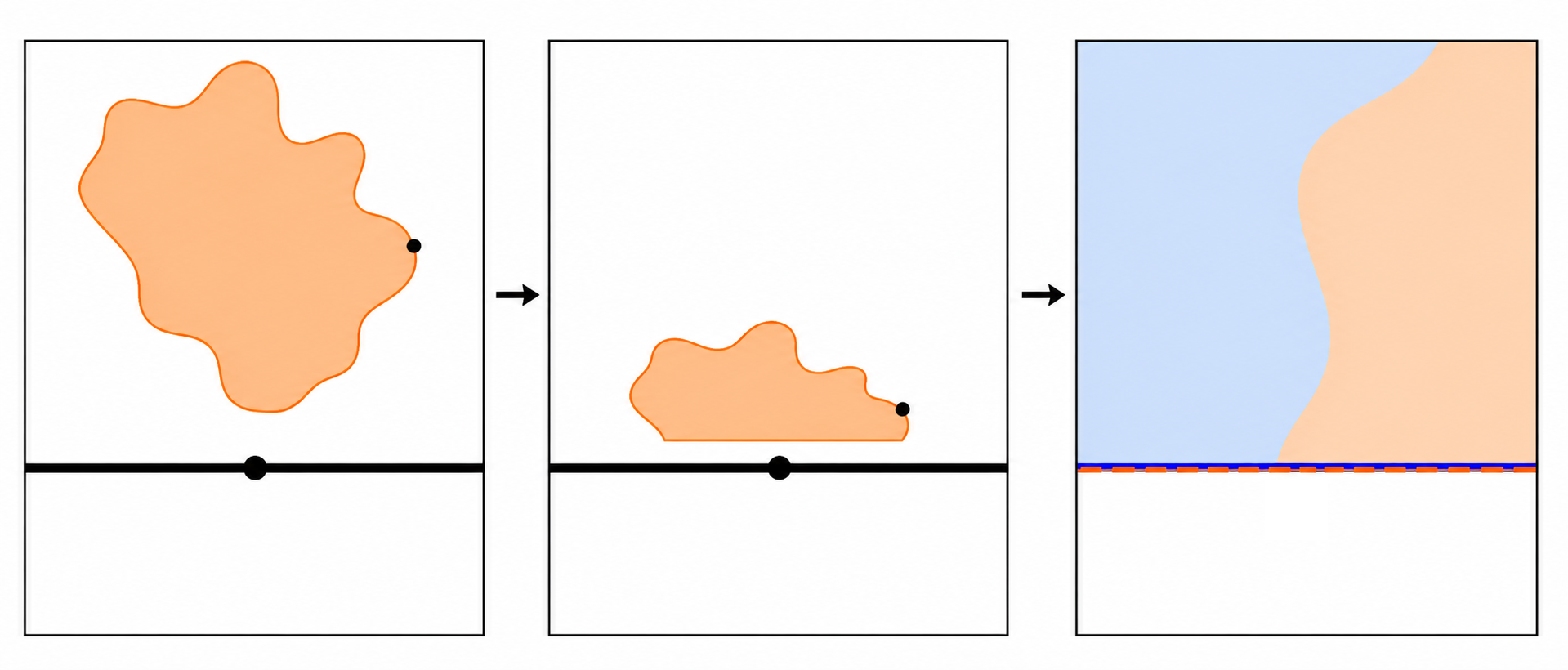}
  \caption{Second blow-down and the double-wall contradiction.}
  \label{old-fig:double-wall}
\end{figure}
\FloatBarrier

The measure identity in \eqref{old-eq:joint-limits} is precisely
\eqref{old-eq:double-wall-identity}.  Taking $\Gamma=P\cap B_1$ in
\cref{old-lem:double-wall} contradicts \eqref{old-eq:L-full-trace}.  Hence
$\mathcal L^{n+1}(K)=0$ and $F=H$ almost everywhere.
\end{proof}

\section{Proof of the minimizing halfspace theorem}
\label{old-sec:global-proof}

In this section, we complete the proof of the minimizing halfspace theorem by
taking a blow-down at infinity.  One-sided confinement places every limit in
the supporting halfspace, and support persistence forces
contact with its boundary; wall rigidity therefore identifies the limit with
the halfspace itself.  Strict flat convergence and thin-slab regularity then
produce single graphical regions whose radii tend to infinity.  Every fixed
boundary point eventually lies in these regions, forcing a common normal and,
ultimately, a single affine hyperplane.

\begin{proof}[Proof of \cref{old-thm:minimizing-halfspace}]
As above, apply an orthogonal map $Q$ to the coordinates and replace the
integrand by $\Phi_Q(\xi):=\Phi(Q^{-1}\xi)$.  Relabel the transformed set and
integrand as $E$ and $\Phi$, and then translate.  Set
\begin{equation}\label{old-eq:final-halfspace-notation}
 H:=\{x_{n+1}>0\},\qquad P:=\partial H,\qquad \nu_H:=-e_{n+1}.
\end{equation}
Suppose that the supporting halfspace is $\overline H$.  Since
$D\chi_E=0$ in the connected open halfspace
$\{x_{n+1}<0\}$, $\chi_E$ is constant there.  If necessary, replace $E$ by
$E^c$ and $\Phi$ by $\Phi^-$ so that $E\subset\overline H$ and hence
$E\subset H$ almost everywhere.  Indeed, $\nu_{E^c}=-\nu_E$ and therefore
\begin{equation}\label{old-eq:complement-reversal}
 P_{\Phi^-}(E^c;W)=P_\Phi(E;W)
\end{equation}
for every relatively compact open set $W$.  Thus this replacement preserves
global minimality, as well as the support in
\eqref{old-eq:one-sided-support}.

Choose $p\in\supp|D\chi_E|$ and translate $p$ to the origin.  For some
$a\ge0$ we then have
\begin{equation}\label{old-eq:translated-supporting-halfspace}
  E\subset\{x_{n+1}\ge-a\}.
\end{equation}
Let $R_j>0$ satisfy $R_j\to\infty$, and set
\begin{equation}\label{old-eq:first-blowdown}
  E_j:=R_j^{-1}E.
\end{equation}
Then $E_j\subset\{x_{n+1}\ge-a/R_j\}$.  The translated supporting plane has
distance $a/R_j$ from the origin.
By \cref{old-prop:DPM-compactness}, a subsequence converges to a global
minimizer $F\subset H$; support persistence gives
$0\in\supp|D\chi_F|$.  The distinction between wall contact and the later
single-sheet conclusion is illustrated in \cref{old-fig:first-blowdown-wall}.
The last panel records the remaining logical possibility at this stage: full
wall trace does not by itself exclude an interior defect
$K=H\setminus F$, nor does it imply global graphicality.

\begin{figure}[htbp]
  \centering
  \includegraphics[width=.98\textwidth]{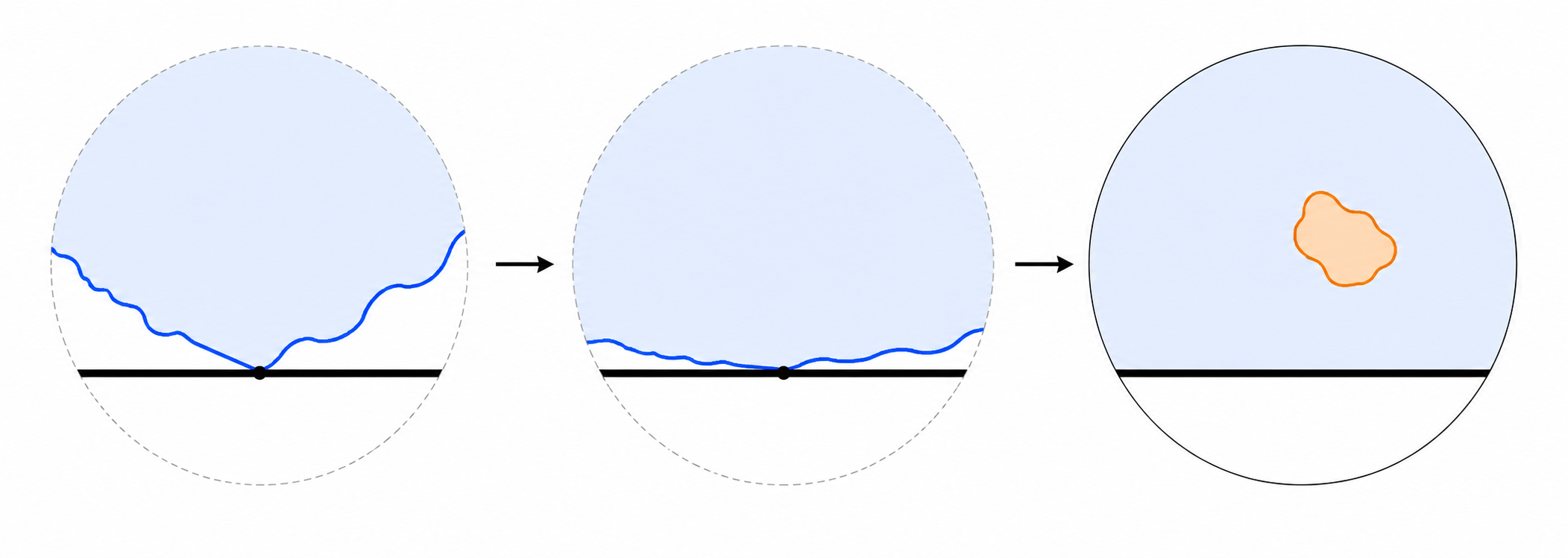}
  \caption{First blow-down and wall contact.}
  \label{old-fig:first-blowdown-wall}
\end{figure}
\FloatBarrier

Proposition
\ref{old-prop:wall-rigidity} gives
\begin{equation}\label{old-eq:flat-blowdown}
  F=H.
\end{equation}

By \cref{old-prop:set-to-current}, the oriented boundary currents of the
sets $E_j$ are absolutely $\Phi$-mass minimizing.  Hence the three
convergences in \eqref{old-eq:DPM-convergence} are exactly the hypotheses of
\cref{old-prop:strict-flat-graphical}.  Scaling its conclusion back to the
original variables, there is a fixed
$\theta_*=\theta_*(n,\lambda)>0$ such that, for all sufficiently large $j$,
\begin{equation}\label{old-eq:large-single-graph}
  \supp|D\chi_E|\cap\mathcal C_{\theta_*R_j}(\nu_H)
\end{equation}
is one connected $C^2$ graph over $P\cap B_{\theta_*R_j}$.  Its outer normal
converges uniformly to the outer normal $\nu_H$ of the limiting halfspace.  In
particular, if $0<\gamma<\theta_*/2$, its restriction to $B_{\gamma R_j}$ is a single
graph as well.  In
\cref{old-fig:expanding-graphs}, the first panel superimposes several members
of this rescaled sequence only to display convergence; they are not multiple
sheets of one boundary.  The middle panel represents the single graph at a
fixed rescaled scale, and the last panel represents the same original
boundary in expanding graphical cylinders.

\begin{figure}[htbp]
  \centering
  \includegraphics[width=.88\textwidth]{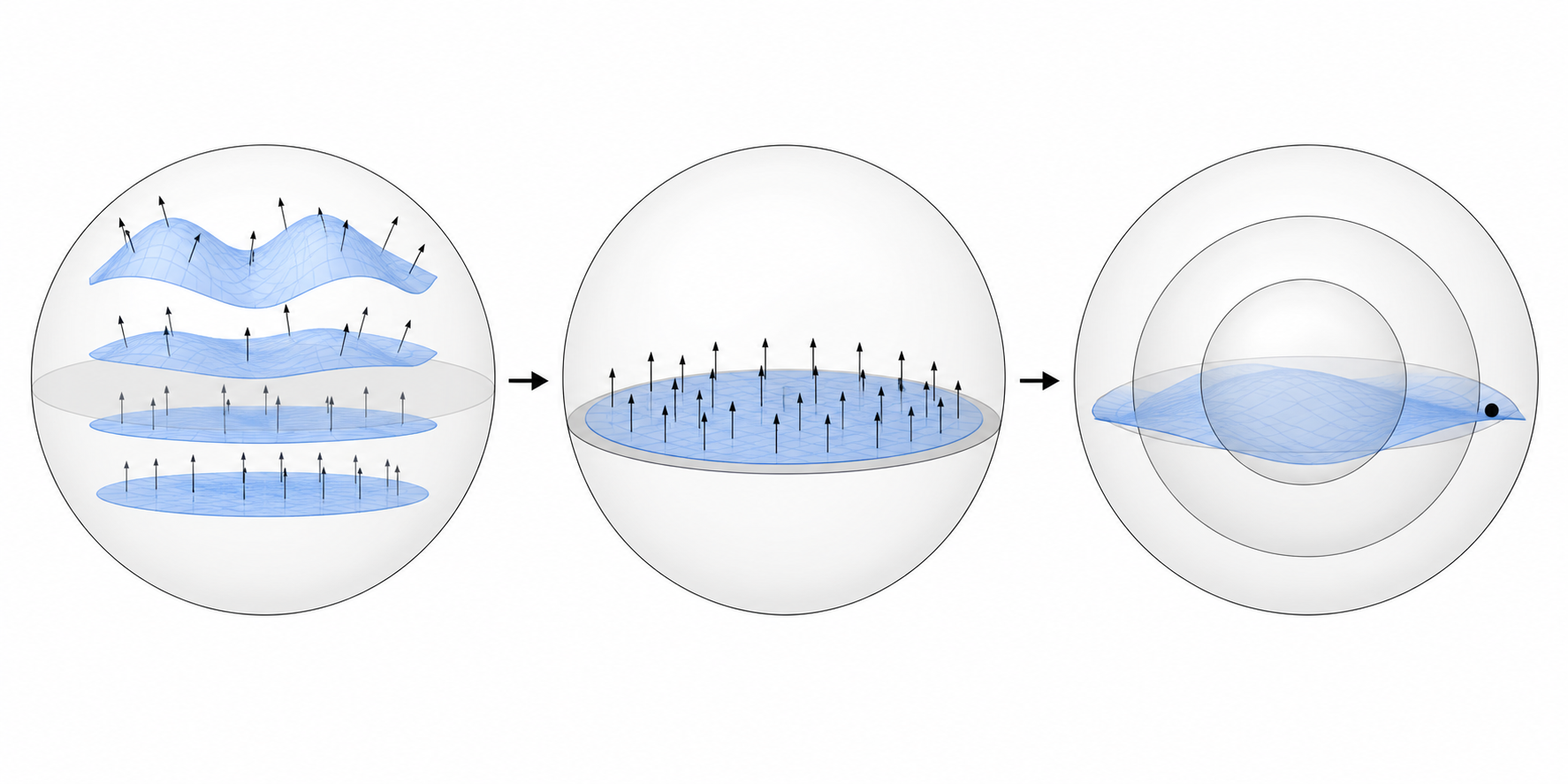}
  \caption{Flat blow-down and expanding graphical regions.}
  \label{old-fig:expanding-graphs}
\end{figure}
\FloatBarrier

Fix any $q\in\supp|D\chi_E|$.  For all sufficiently large $j$, the expanding
graph \eqref{old-eq:large-single-graph} contains $q$.  Thus $q$ is regular.
Dilation does not change its normal, while the graph normals converge
uniformly to $\nu_H$; hence the fixed vector $\nu_E(q)$ equals $\nu_H$.
Since $q$ was arbitrary, every point of the support is regular and has the
same normal.

Fix $q$ once more and choose $j$ so large that one connected graph in
\eqref{old-eq:large-single-graph} contains both the origin and $q$.  Its
graphing function has zero gradient everywhere, so the two points have the
same height.  Hence the support is contained in $P$.  Conversely, for every
$y\in P$, an expanding graph in \eqref{old-eq:large-single-graph} eventually
covers $y$ and supplies a point of the support above it; the preceding
conclusion forces that point to be $y$.  Therefore the support is exactly
$P$.  The
characteristic function is constant on the two components of the complement
of this hyperplane.  The lower component contains
$\{x_{n+1}<-a\}\subset E^c$, so its constant is zero.  Since the boundary is
nonempty, the constant on the upper component is one.  Thus $E$ agrees almost
everywhere with the corresponding halfspace.  Undoing the rigid motion,
translation, and possible complement completes the proof.
\end{proof}


\begingroup
\footnotesize
\begin{thebibliography}{WWXZ26}

\bibitem[AFP00]{AmbrosioFuscoPallara2000}
Ambrosio, L.; Fusco, N.; Pallara, D.
\emph{Functions of bounded variation and free discontinuity problems}.
Oxford Math. Monogr., The Clarendon Press, Oxford Univ. Press,
New York, 2000, xviii+434 pp.

\bibitem[All74]{Allard1974}
Allard, W.~K.
A characterization of the area integrand.
In: \emph{Symposia Mathematica, Vol.~XIV}, Academic Press,
London, 1974, 429--444.

\bibitem[Ber10]{Bergner2010}
Bergner, M.
A halfspace theorem for proper, negatively curved immersions.
\emph{Ann. Global Anal. Geom.} \textbf{38} (2010), 191--199.

\bibitem[BDGG69]{BombieriDeGiorgiGiusti1969}
Bombieri, E.; De Giorgi, E.; Giusti, E.
Minimal cones and the Bernstein problem.
\emph{Invent. Math.} \textbf{7} (1969), 243--268.

\bibitem[BG72]{BombieriGiusti1972}
Bombieri, E.; Giusti, E.
Harnack's inequality for elliptic differential equations on minimal surfaces.
\emph{Invent. Math.} \textbf{15} (1972), 24--46.

\bibitem[DDRG18]{DePhilippisDeRosaGhiraldin2018}
De Philippis, G.; De Rosa, A.; Ghiraldin, F.
Rectifiability of varifolds with locally bounded first variation with respect
to anisotropic surface energies.
\emph{Comm. Pure Appl. Math.} \textbf{71} (2018), 1123--1148.

\bibitem[DPM15]{DePhilippisMaggi2015}
De Philippis, G.; Maggi, F.
Regularity of free boundaries in anisotropic capillarity problems and the
validity of Young's law.
\emph{Arch. Ration. Mech. Anal.} \textbf{216} (2015), 473--568.

\bibitem[DRT22]{DeRosaTione2022}
De Rosa, A.; Tione, R.
Regularity for graphs with bounded anisotropic mean curvature.
\emph{Invent. Math.} \textbf{230} (2022), 463--507.

\bibitem[DMYZ25]{DuMooneyYangZhu2025}
Du, W.; Mooney, C.; Yang, Y.; Zhu, J.
A half-space Bernstein theorem for anisotropic minimal graphs.
\emph{J. Eur. Math. Soc. (JEMS)} (2025), published online first,
DOI: 10.4171/JEMS/1695.

\bibitem[DY24]{DuYang2024}
Du, W.; Yang, Y.
Flatness of anisotropic minimal graphs in $\R^{n+1}$.
\emph{Math. Ann.} \textbf{390} (2024), 4931--4949.

\bibitem[EW22]{EdelenWang2022}
Edelen, N.; Wang, Z.
A Bernstein-type theorem for minimal graphs over convex domains.
\emph{Ann. Inst. H. Poincar\'e C Anal. Non Lin\'eaire} \textbf{39} (2022),
749--760.

\bibitem[Fed69]{Federer1969}
Federer, H.
\emph{Geometric measure theory}.
Grundlehren Math. Wiss., vol.~153, Springer-Verlag,
New York, 1969, xiv+676 pp.

\bibitem[FM11]{FigalliMaggi2011}
Figalli, A.; Maggi, F.
On the shape of liquid drops and crystals in the small mass regime.
\emph{Arch. Ration. Mech. Anal.} \textbf{201} (2011), 143--207.

\bibitem[FMP10]{FigalliMaggiPratelli2010}
Figalli, A.; Maggi, F.; Pratelli, A.
A mass transportation approach to quantitative isoperimetric inequalities.
\emph{Invent. Math.} \textbf{182} (2010), 167--211.

\bibitem[HM90]{HoffmanMeeks1990}
Hoffman, D.; Meeks, W.~H., III.
The strong halfspace theorem for minimal surfaces.
\emph{Invent. Math.} \textbf{101} (1990), 373--377.

\bibitem[Jen61]{Jenkins1961}
Jenkins, H.~B.
On two-dimensional variational problems in parametric form.
\emph{Arch. Ration. Mech. Anal.} \textbf{8} (1961), 181--206.

\bibitem[KP05]{KoisoPalmer2005}
Koiso, M.; Palmer, B.
Geometry and stability of surfaces with constant anisotropic mean curvature.
\emph{Indiana Univ. Math. J.} \textbf{54} (2005), no.~6, 1817--1852.

\bibitem[KP08]{KoisoPalmer2008}
Koiso, M.; Palmer, B.
Rolling construction for anisotropic Delaunay surfaces.
\emph{Pacific J. Math.} \textbf{234} (2008), no.~2, 345--378.

\bibitem[KP11]{KuhnsPalmer2011}
Kuhns, C.; Palmer, B.
Helicoidal surfaces with constant anisotropic mean curvature.
\emph{J. Math. Phys.} \textbf{52} (2011), no.~7,
073506, 14 pp.

\bibitem[Mag12]{Maggi2012}
Maggi, F.
\emph{Sets of finite perimeter and geometric variational problems:
An introduction to geometric measure theory}.
Cambridge Stud. Adv. Math., vol.~135,
Cambridge Univ. Press, Cambridge, 2012.

\bibitem[Moo22]{Mooney2022}
Mooney, C.
Entire solutions to equations of minimal surface type in six dimensions.
\emph{J. Eur. Math. Soc. (JEMS)} \textbf{24} (2022), 4353--4361.

\bibitem[MY21]{MooneyYang2021}
Mooney, C.; Yang, Y.
A proof by foliation that Lawson's cones are $A_\Phi$-minimizing.
\emph{Discrete Contin. Dyn. Syst.} \textbf{41} (2021), 5291--5302.

\bibitem[MY24]{MooneyYang2024}
Mooney, C.; Yang, Y.
The anisotropic Bernstein problem.
\emph{Invent. Math.} \textbf{235} (2024), 211--232.

\bibitem[Sav10]{Savin2010}
Savin, O.
Minimal surfaces and minimizers of the Ginzburg--Landau energy.
In: Farina, A.; Valdinoci, E. (eds.),
\emph{Symmetry for elliptic PDEs},
Contemp. Math., vol.~528, Amer. Math. Soc.,
Providence, RI, 2010, pp.~43--57.

\bibitem[SSA77]{SchoenSimonAlmgren1977}
Schoen, R.; Simon, L.; Almgren, F.~J., Jr.
Regularity and singularity estimates on hypersurfaces minimizing parametric
elliptic variational integrals.
\emph{Acta Math.} \textbf{139} (1977), no.~3--4, 217--265.

\bibitem[Sim77]{Simon1977}
Simon, L.
On some extensions of Bernstein's theorem.
\emph{Math. Z.} \textbf{154} (1977), 265--273.

\bibitem[Sim87]{Simon1987}
Simon, L.
A strict maximum principle for area minimizing hypersurfaces.
\emph{J. Differential Geom.} \textbf{26} (1987), no.~2, 327--335.

\bibitem[Sim68]{Simons1968}
Simons, J.
Minimal varieties in Riemannian manifolds.
\emph{Ann. of Math. (2)} \textbf{88} (1968), 62--105.

\bibitem[SW89]{SolomonWhite1989}
Solomon, B.; White, B.
A strong maximum principle for varifolds that are stationary with respect
to even parametric elliptic functionals.
\emph{Indiana Univ. Math. J.} \textbf{38} (1989), no.~3, 683--691.

\bibitem[WWXZ26]{WangWeiXiaZhang2026}
Wang, G.; Wei, W.; Xia, C.; Zhang, X.
A half-space Liouville theorem for anisotropic minimal graph with free boundary.
\emph{Arch. Ration. Mech. Anal.} \textbf{250} (2026),
Paper No.~42, 34 pp.

\end{thebibliography}
\endgroup
\end{document}